\documentclass[a4paper]{amsart}

\usepackage{amsthm}
\usepackage{amsmath}
\usepackage{amssymb}
\usepackage{latexsym}
\usepackage{enumerate}
\usepackage{graphicx}% Include figure files
\usepackage[utf8]{inputenc}

\usepackage{epsfig} %inclosure fig EPS
\usepackage{pgf}
\usepackage{tikz}
\usepackage{float}
\usepackage{dsfont}
\usepackage{times}
\renewcommand{\thetheoremName}

\newtheorem{proposition[[]]}[theoremName]{Proposition G}

\newtheorem{theorem}{Theorem}[section]
\newtheorem{lemma}[theorem]{Lemma}

\newtheorem{proposition}[theorem]{Proposition}
\newtheorem{corollary}[theorem]{Corollary}

\theoremstyle{definition}
\newtheorem{definition}[theorem]{Definition}
\newtheorem{example}[theorem]{Example}

\newtheorem{remark}{Remark}

\numberwithin{equation}{section}

\newcommand{\tr}{\operatorname{tr}}

\newcommand{\erre}{\mathbb{R}}

\newcommand{\ene}{N}
\newcommand{\Vol}{\operatorname{Vol}}

\begin{document}

  \title[Parabolicity and properness of self-similar solutions of MCF and IMCF]{Parabolicity, Brownian escape rate and properness of self-similar solutions of the direct and inverse mean curvature flow}

\author[V. Gimeno]{Vicent Gimeno*}
\address{Departament de Matem\`{a}tiques- IMAC,
Universitat Jaume I, Castell\'{o}, Spain.}
\email{gimenov@uji.es}
\author[V. Palmer]{Vicente Palmer**}
\address{Departament de Matem\`{a}tiques- INIT,
Universitat Jaume I, Castellon, Spain.}
\email{palmer@mat.uji.es}
%\thanks{}
\thanks{* Work partially supported by the Research Program of University Jaume I Project P1-1B2012-18, and DGI -MINECO grant (FEDER) MTM2013-48371-C2-2-P}
\thanks{**Work partially supported by the Research Program of University Jaume I Project P1-1B2012-18,  DGI -MINECO grant (FEDER) MTM2013-48371-C2-2-P, and Generalitat Valenciana Grant PrometeoII/2014/064 }

%\authorrunning{Short form of author list} % if too long for running head

%xxxxxxxxxxxxxxxxx  Mod 2.5 xxxxxx

\begin{abstract} 
 We study some potential theoretic properties of homothetic solitons $\Sigma^n$ of the MCF and the IMCF. Using the analysis of the extrinsic distance function defined on these submanifolds in $\erre^{n+m}$, we observe similarities and differences in the geometry of solitons in both flows. In particular, we show that parabolic MCF-solitons $\Sigma^n$ with $n>2$ are self-shrinkers and that parabolic IMCF-solitons of any dimension are self-expanders. We have studied too the geometric behavior of parabolic MCF and IMCF-solitons confined in a ball, the behavior of the Mean Exit Time function for the Brownian motion defined on $\Sigma$ as well as a classification of properly immersed MCF-self-shrinkers with bounded second fundamental form, following the lines of \cite{CaoLi}.
\end{abstract}

\maketitle
\tableofcontents

%%%%%%%%%%%%%%%%%%%%%%%%%
\section{Introduction}\label{Intr}
%%%%%%%%%%%%%%%%%%%%%%%
The potential theory on a complete manifold is mainly devoted to the study of harmonic (or subharmonic) functions defined on it, and, more generally, to the study of the relation among the geometry of the manifold and the properties of the solutions of some distinguished PDEs raised using the Laplace-Beltrami operator, such us Laplace and Poisson equations. The interplay between geometric information, (encoded in the form of bounds for the curvature, for example) and functional theoretic properties, (such as the existence of bounded harmonic or subharmonic functions) constitutes a rich arena at the crossroads of Functional Analysis, Differential Geometry and PDEs theory where the problems we are going to study are placed. To address these problems, we will add in this paper the point of view of submanifold theory, in relation with some distinguished submanifolds in the Euclidean space. In particular,
we are going to focus in the study of the parabolicity of homothetic solitons for the Mean Curvature Flow and for the Inverse Mean Curvature Flow and the relation of this concept with the geometry of these submanifolds. We are going to apply the same technique, namely, the analysis of the extrinsic distance defined on the submanifold, on MCF and IMCF solitons, in order to highlight similarities and differences among them.

We recall that a non-compact, complete $n$-dimensional manifold $M^n$ is {\em parabolic} if and only if every subharmonic,($\Delta u\geq 0$ when $u \in C^2(M)$), and bounded ($\sup_{M} u =u^* < \infty$) continuous function $u: M \rightarrow \erre$ defined on it is constant. If such non-constant  function exists, then  $M$ is {\em non-parabolic}. This functional property holds in compact manifolds as a direct application of the strong Maximum Principle, so parabolicity can be viewed as generalization of compactness. 

In fact, and if we modify slightly our point of view, parabolicity can be viewed as a stronger version of the following {\em weak Maximum Principle}: given $M$ a (not necessarily complete) Riemannian manifold, it satisfies the weak Maximum Principle if an only if for any bounded function $u \in C^2(M)$ with $\sup_{M} u =u^* < \infty$, there exists a sequence of points $\{x_k\}_{k \in \ene} \subseteq M$ such that $u(x_k) > u^*-\frac{1}{k}$ and $\Delta u(x_k) < \frac{1}{k}$, (see \cite{AMR}).

Let us consider now an isometric immersion $X: \Sigma \to \erre^{n+m}$ of the manifold $\Sigma^n$ in $\erre^{n+m}$. A question that arises naturally when studying the parabolicity of $\Sigma$ consists in to obtain a geometric description of this potential theoretic property, relating it, for example, with the behavior of its mean curvature. In this sense, when the dimension of the  submanifold is $n=2$, minimality does not imply parabolicity nor non-parabolicity: some minimal surfaces in $\erre^3$ are parabolic, (e.g. Costa's surface, Helicoid, Catenoid), while some others (like P-Schwartz surface or Scherk doubly periodic surface,) are non-parabolic.

However, something can be said in this context. In particular, we have, by one hand, that
{\em complete and minimal isometric immersions $ \varphi: \Sigma^2 \rightarrow \erre^n$ included in a ball $\varphi(\Sigma)\subseteq B^n_R$ are non-parabolic}. The proof of this theorem follows from the fact that coordinate functions 
$x_i :\Sigma \to \erre$ are harmonic, bounded in $\varphi(\Sigma)\subseteq B^{n}_R$ and non-constant. Recall that in the paper \cite{Na}, N. Nadirashvili constructed a complete (non-proper) immersion of a minimal disk into the unit ball in $\erre^3$. 

On the other hand, and when the dimension of the submanifold is bigger or equal than $3$, we have that {\em complete and minimal proper isometric immersions $ \varphi: \Sigma^n \rightarrow \erre^{n+m}$ with $n \geq 3$ are non-parabolic}
(see \cite{MP}). The proof in this case is based on obtaining bounds for the {\em capacity} at infinity of a suitable precompact set in the submanifold.

Since solitons for MCF and IMCF satisfy a geometric condition on its mean curvature, namely, equations (\ref{geocon}) and (\ref{geoconInverse}) in Definitions \ref{solitonMCF} and \ref{solitonIMCF} respectively,
and  inspired by the results above mentioned,  it could be interesting to establish a geometric description of parabolicity of a complete and non-compact soliton for the MCF and IMCF, and to study the behavior of parabolic solitons confined in a ball. To do that, we have used the analysis of the Laplacian of radial functions depending on the extrinsic distance, and Theorem \ref{teo0}, (see \cite{AMR}), where it is proved that parabolicity implies the weak Maximum Principle alluded above.

In what follows, we are going to give an account of our main results concerning these and other related questions.

In Theorem \ref{teorNecMCF},  we prove that parabolic solitons for the MCF with dimension $n \geq 3$ are self-shrinkers and in Corollary \ref{cor7}, we prove that self-expanders for the MCF are non-parabolic. In this  line and using the techniques mentioned before, we have proved in Theorem \ref{teorNecIMCF} that parabolic solitons for the IMCF are self-expanders, and that self-shrinkers for the IMCF with $n \geq 2$, and self-expanders for the IMCF with $n \geq 3$ and velocity $C > \frac{1}{n-2}$ are non-parabolic, (Corollary \ref{cor5}).

Another line of research that we mentioned above is the study of the behavior of solitons included in a ball or in a half-space containing the origin. We can find in the literature several works dealing with this question, for example the paper \cite{EC}, where it is extended the Hoffman-Meeks Halfspace Theorem to properly immersed self-shrinkers for the MCF, or the work \cite{PiRi}, where some classification results for self-shrinkers for MCF are presented, assuming some restrictions on  the norm of its second fundamental form, and considering that the self-shrinker is confined in a ball or a generalized cylinder and it has bounded mean curvature. 

Our results in this line of research are Theorem \ref{teorConfMCF}, where it is proved that complete and parabolic self-shrinkers for the MCF confined in the ball $B^{n+m}(\sqrt{\frac{n}{\lambda}})$ centered at $\vec{0} \in \erre^{n+m}$ must be compact minimal submanifolds of the sphere $S^{n+m-1}(\sqrt{\frac{n}{\lambda}})$ and, as a corollary, that the only complete and connected parabolic self-shrinkers for the MCF with codimension $1$ confined in the ball $B^{n+m}(\sqrt{\frac{n}{\lambda}})$ are the spheres of radius $\sqrt{\frac{n}{\lambda}}$. Moreover, we have proved that there are not complete and non-compact parabolic self-expanders for MCF confined in a ball of any radius, (Theorem \ref{noselfexpanders}). Concerning solitons for the IMCF we have proved in Theorem \ref{teorConfIMCF2} that complete and non-compact  parabolic solitons confined in a $R$-ball are compact minimal submanifolds of a sphere of radius less or equal than $R$.

In regard to classification results using bounds for the norm of the second fundamental form, in the paper 
 \cite{CaoLi}, the authors obtained a classification theorem for complete self-shrinkers of MCF without boundary and with polynomial volume growth satisfying that the squared norm of its second fundamental form is less or equal than $1$, ($\lambda$ in the case we consider $\lambda$-self-shrinkers). Using the Mean Exit Time function, (whose behavior is closely related with the notion of parabolicity) defined on the extrinsic balls of the solitons, we have obtained some classification results for them. In particular, in first place, (Theorem \ref{isopShri}), we have established an isoperimetric inequality satisfied by properly immersed MCF-self-shrinkers $X: \Sigma^n \rightarrow \erre^{n+m}$ and, from this result we have shown: first, that the properly immersed self-shrinkers confined in the $\sqrt{\frac{n}{\lambda}}$-ball $B^{n+m}(\sqrt{\frac{n}{\lambda}})$ or included in the complementary set $\erre^{n+m} \setminus B^{n+m}(\sqrt{\frac{n}{\lambda}})$ must be compact minimal submanifolds of the sphere $S^{n+m-1}(\sqrt{\frac{n}{\lambda}})$ (Theorem \ref{ballcomp}), and secondly, (Theorem \ref{teo7.7v24}), that, if in addition the squared norm of the second fundamental form of these $\lambda$-self-shrinkers is bounded by the quantity $\frac{5}{3}\lambda$,  then they must be the sphere $S^{n+m-1}(\sqrt{\frac{n}{\lambda}})$, or, alternatively, this sphere separates the soliton into two parts. We present finally a characterization of IMCF-solitons in terms of the Mean Exit Time function defined on its extrinsic balls, (Theorem \ref{solcarac}).

\subsection{Outline of the paper} The structure of the paper is as follows:

In the preliminaries, Section \S \ref{prelim}, subsection \S \ref{exdist}, we recall the preliminary concepts and properties of extrinsic distance function. In subsection \S \ref{parab} it is presented and studied the notion of {\em parabolicity}, together a result due to Alias, Mastrolia and Rigoli, which extends the maximum principle to complete and non-compact manifolds that shall be widely used along the paper.  We finish the preliminaries defining the solitons for the MCF and IMCF, (subsection \S \ref{solitons}) and relating them with the minimal spherical immersions, (subsection \S \ref{solImm}).

We shall prove Theorem \ref{teorNecMCF} and Corollary \ref{cor7}  in  subsection \S \ref{necondMCF} of Section \S \ref{geodescMCF}, and  Theorem \ref{teorNecIMCF} and Corollary \ref{cor5} in subsection \S \ref{sufcondMCF} of Section \S \ref{geodescMCF}. In Section \S \ref{exam} we check some of the parabolicity and non-parabolicity criteria we have proved on some examples. In Section \S \ref{confined} we shall study solitons confined in a ball: we prove Theorem \ref{teorConfMCF}  in subsection  \S \ref{confinedMCF}, obtaining Corollary  \ref{cor10} and Theorem \ref{noselfexpanders}. In subsection \S \ref{confinedIMCF} we have proved Theorem \ref{teorConfIMCF2} and Corollaries \ref{cor9} and \ref{cor11}.

 Finally, in \S \ref{meanexit}, subsection \S \ref{volumeSS}, the isoperimetric inequality,  Theorem \ref{isopShri}, is proved. Then, in subsection \S \ref{clasiSS} we have the classification theorem \ref{teo7.7v24}. The characterization Theorem \ref{solcarac} is given in subsection \S \ref{volumeS} of Section \S \ref{meanexitInverse} and an isoperimetric inequality for  IMCF solitons is presented in Theorem \ref{isopSol} in subsection \S \ref{volumS}.
%%%%%%%%%%%%%%%%%%%
\section{Preliminaries}\label{prelim}\
%%%%%%%%%%%%%%
%%%%%%%%%%%%%%%%%%%%%%%%%%%%%%%%%%
\subsection{The extrinsic distance function}\label{exdist}
%%%%%%%%%%%%%%%%%%%%%%%%%%%
Let $X:\Sigma^n\to \erre^{n+m}$ be a complete isometric immersion into the Euclidean space $\erre^{n+m}$. The \emph{extrinsic distance function of $X$} to the origin $\vec{0} \in \erre^{n+m}$ is given by
$$
r:\Sigma\to \erre,\quad r(p)={\rm dist}_{\erre^{n+m}}\left(\vec{0},\, X(p)\right)=\Vert X(p)\Vert. 
$$
\noindent In the above equality,  $\Vert \, , \Vert$ denotes the norm of vectors in $\erre^{n+m}$ induced by the usual metric $g_{\erre^{n+m}}$. The gradients of $r(x)={\rm dist}_{\erre^{n+m}}(\vec{0},\, x )$ in $\erre^{n+m}$ and in  $\Sigma$ are
denoted by $\nabla^{\erre^{n+m}} r$ and $\nabla^\Sigma r$,
respectively. Then we have
the following basic relation, 
\begin{equation}\label{radiality}
\nabla^{\erre^{n+m}} r = \nabla^\Sigma r +(\nabla^{\erre^{n+m}} r)^\bot \,\,\,\text{on}\,\,\Sigma
\end{equation}
where $(\nabla^{\erre^{n+m}} r)^\bot(X(x))=\nabla^\bot r(X(x))$ is perpendicular to
$T_{x}\Sigma$ for all $x\in \Sigma$.

\begin{definition}\label{ExtBall}
Let $X:\Sigma^n\to \erre^{n+m}$ be a complete isometric immersion into the Euclidean space $\erre^{n+m}$. We denote the {\em{extrinsic metric balls}} of radius $R >0$ and center $\vec{0} \in \erre^{n+m}$  by
$D_R$. They are defined as  the subset of $\Sigma$:
$$
D_R=\{x\in \Sigma : r(x)< R\}=\{x\in \Sigma : X(x) \in B^{n+m}_{R}(\vec{0})\}=X^{-1}(B^{n+m}_{R}(\vec{0}))
$$
where $B^{n+m}_{R}(\vec{0})$ denotes the open geodesic ball
of radius $R$ centered at the pole $\vec{0} \in \erre^{n+m}$. Note that the set $X^{-1}(\vec{0})$ can be the empty set.
\end{definition}
\begin{remark}\label{theRemk0}
When the immersion $X$ is proper, the extrinsic domains $D_R$
are precompact sets, with smooth boundaries $\partial D_R$. The assumption on the smoothness of
$\partial D_{R}$ makes no restriction. Indeed, 
the distance function $r$ is smooth in $\erre^{n+m} - \{\vec{0}\}$ 
since $\vec{0}$ is a pole of $\erre^{n+m}$. Hence
the composition $r\vert_\Sigma$ is smooth in $\Sigma$ and consequently the
radii $R$ that produce non-smooth boundaries
$\partial D_R$ have $0$-Lebesgue measure in $\mathbb{R}$ by
Sard's theorem and the Regular Level Set Theorem.
\end{remark}
\begin{remark}
Along the paper, we shall denote as $S^{n+m-1}(R)$ and as $B^{n+m}(R)$ or $B^{n+m}_{R}(\vec{0})$ the spheres and the balls centered at $\vec{0}$ in $\erre^{n+m}$. In the classification results, (as Corollaries \ref{cor10} and \ref{cor11}, or Theorem \ref{teo7.7v24}), we are also using this notation to denote the $n$-dimensional $R$-spheres $\mathbb{S}^{n}(R)$ considered as Riemannian manifolds, where the center it is not relevant. Another place where the center of the balls and spheres is not relevant is in the Poisson problem (\ref{poissoneuclidean}). In all the cases we are using the same notation, and the relevance or not of the center and if we are considering the spheres immersed or not will be clear from the context.
\end{remark}

A technical result which we will use is the following: 

\begin{lemma}\label{lemma}
Let $X:\Sigma^n\to \erre^{n+m}$ be a complete isometric immersion into the Euclidean space $\erre^{n+m}$. Let $r:\Sigma\to \erre,\quad r(p)=dist_{\erre^{n+m}}(X(p), \vec{0})=\Vert X(p)\Vert$ the extrinsic distance of the points in $\Sigma$ to the origin $\vec{0} \in \erre^{n+m}$.  Given any function $F:\erre\to \erre$, we have that
\begin{equation}\label{eq2.3}
\begin{array}{ccc}
 \triangle^{\Sigma} F(r(x))&=&\left(\frac{F''(r(x))}{r^2(x)}-\frac{F'(r(x))}{r^3(x)}\right)\Vert X^T\Vert^2\\
 & & +\frac{F'(r(x))}{r(x)}\left(n+\langle X, \vec{H}\rangle\right)
\end{array}
 \end{equation}
where $X^T$ denotes here the tangential component of $X$ with respect to $X(\Sigma)$ and $\vec{H}$ denotes the mean curvature vector field of $\Sigma$.
\end{lemma}

%%%%%%%%%%%%%%%%%%%%%%%%%
\subsection{Parabolicity and capacity estimates}\label{parab}
%%%%%%%%%%%%%%%%%%%%%%%%%%

Parabolicity extends the maximum principle to complete and non-compact parabolic manifolds in the following way, (see \cite{AMR}): 
\begin{theorem}\label{teo0}
Let $M$ be a complete non compact and parabolic Riemannian manifold. Then for each $u\in C^2(\Sigma)$, $\sup u<\infty$, $u$ nonconstant on $\Sigma$, there exists a sequence $\{x_k\}\subset \Sigma$
 such that $u(x_k)>\sup u-\frac{1}{k}$, $\triangle u(x_k)<0$, $\forall k\in \mathbb{N}$.
 \end{theorem}
To relate this functional property with the geometry of the underlying manifold, we shall establish bounds for the {\em capacity} of $M$. When $\Omega\subset M$ is precompact, it can be proved, (see \cite{Gri}), that  the {\em capacity} of the compact $K$ in $\Omega$ is given as the following  integral:
\[
{\rm cap}(K,\Omega)=\int_{\Omega}\left\Vert \nabla\phi\right\Vert ^{2} dV_g=\int_{\partial K}\Vert \nabla \phi\Vert  d\mu
\]
where $\phi$ is the solution of the Laplace equation on $\Omega -K$ with Dirichlet boundary values:
\begin{equation}\label{laplace}
\left\{
\begin{array}
[c]{c}%
\Delta u=0\\
u\mid_{\partial K}=1\\
u\mid_{\partial\Omega}=0
\end{array}
\right.%
\end{equation}
Moreover, for any compact $K\subset \Sigma$ and any
open set $G\subset \Sigma$ containing $K$, we have
\begin{equation}\label{eq2.1}
 {\rm cap}(K, \Sigma)\leq {\rm cap}(K,G)
\end{equation}
The relation among  capacity and parabolicity is given by the following result, (see \cite{Gri}): 
\begin{theorem}\label{theorGrig}
 Let $(M,g)$ be a Riemannian manifold. $M$ is
parabolic iff  $M$ has \emph{zero capacity}, i.e., there exists a non-empty
precompact $D \subseteq M$ such that ${\rm cap}(D,M)=0$.
\end{theorem}
On the other hand, it can be proved that given $K \subset M$  a (pre)compact subset of $M$, if  we consider $\{\Omega_i\}_{i=1}^{\infty}$ an exhaustion of $M$ by nested and precompact sets, such that $K \subseteq \Omega_i$ for some $i$, then
the capacity of $K$ in all the manifold, (the {\em capacity at infinity} ${\rm cap}(K, M)={\rm cap}(K)$) is given as the following  limit:
\[
{\rm cap}(K, M)=\lim_{i \to \infty}{\rm cap}(K, \Omega_i)
\]
This definition is independent of the exhaustion.
Another result concerning bounds for the capacity of a manifold is following:
\begin{theorem}[ \cite{Gri}]\label{teo2.2}
 Let $\Sigma$ be a complete and non-compact Riemannian manifold. Let $G\subset \Sigma$ be a precompact open set and $K\subset G$ be compact. Suppose that a Lipschitz function
 $u$ is defined in $\overline{G\setminus K}$ such that $u=a$ on $\partial K$ and $u=b$ on $\partial G$ where $a<b$ are real constants. Then,
 \begin{equation}\label{eq2.2}
 {\rm cap}(K,G)\leq \left(\int_a^b \frac{dt}{\underset{\{x\, :\, u(x)=t\}}{\int}\Vert \nabla u(x)\Vert dA(x)}\right)^{-1} 
 \end{equation}
\end{theorem}
To obtain sufficient conditions for parabolicity, we shall apply the following criterion of Has'minskii
\begin{theorem}[\cite{H}]\label{k}
Let $M$ be a Riemannian manifold. If there exists $v: M \to \erre$ superharmonic outside a compact set, and $v(x) \to \infty$ when $x \to \infty$, then $M$ is parabolic
\end{theorem}

%%%%%%%%%%%%%%%%%%%%%%%%
\subsection{Solitons}\label{solitons}
%%%%%%%%%%%%%%%%%%%%%%%%%%
Let $X_0:\Sigma^n\to \erre^{n+m}$ be an isometric immersion of an $n$-dimensional manifold $\Sigma$ into the Euclidean space $\erre^{n+m}$. The evolution of $X_0$ 
by mean curvature flow (MCF) is a smooth one-parameter family of immersions  satisfying
\begin{equation}\label{MCF}
 \left\{\begin{array}{ccc}
         \frac{\partial}{\partial t}X(p,t)&=& \vec{H}(p,t)\,\,\,\forall p \in \Sigma, \,\, \forall \, t \geq 0\\
         X(p,0)&=&X_0(p),\,\,\forall p \in \Sigma\\
        \end{array}
 \right.
\end{equation}
Here, $\vec{H}_t=\vec{H}(\, ,t)$ is the mean curvature vector of the immersion $X_t=X(\, , t)$ \emph{i.e.}, the trace of the second fundamental form $\alpha_t$, ($\vec{H}_t=\tr_{g_t} \alpha_t=\triangle_{g_t} X_t$). Likewise, the evolution 
of the initial immersion $X_0$ by the inverse of the mean curvature flow (IMCF) is a one-parameter family of immersions satisfying
\begin{equation}\label{IMCF}
 \left\{\begin{array}{ccc}
         \frac{\partial}{\partial t}X(p,t)&=&-\frac{\vec{H}(p,t)}{\Vert \vec{H}(p,t)\Vert^2}\,\,\,\forall p \in \Sigma, \,\, \forall \, t \geq 0\\
         X(p,0)&=&X_0(p),\,\,\forall p \in \Sigma\\
        \end{array}
 \right.
\end{equation}
We are going to fix the notions we shall use along the paper, (see \cite{ColdMin} and \cite{Mante} for the definition of soliton).
\begin{definition}\label{solitonMCF}
A complete isometric immersion $X:\Sigma^n\to\erre^{n+m}$ is a $\lambda$-soliton for the MCF with respect $\vec{0} \in \erre^{n+m}$, ($\lambda \in \erre$), if and only if 
\begin{equation}\label{geocon}
\vec{H} =-\lambda X^\perp 
\end{equation}
where $X^\perp$ stands for the normal component of $X$ and $\vec{H}$ is the mean curvature vector of the immersion $X$. 
\end{definition}
\begin{remark}
Note that, if we have a complete isometric immersion $X:\Sigma^n\to\erre^{n+m}$ satisfying the geometric condition (\ref{geocon}), and we consider the family of immersions $X_t=\sqrt{1-2 \lambda t} X$, it is straightforward to check that $\{X_t\}_{t=0}^\infty$ satisfies equation (\ref{MCF}), so $X$ becomes the $0$-slice of the family $\{X_t\}_{t=0}^\infty$ of solutions of equation (\ref{MCF}).
\end{remark}
\begin{definition}
A $\lambda$-soliton for the MCF with respect $\vec{0} \in \erre^{n+m}$ is called a self-shrinker if and only if $\lambda >0$. It is called a self-expander if and only if $\lambda < 0$. 
\end{definition}
\begin{remark}
Note that a complete and minimal immersion $X:\Sigma^n\to\erre^{n+m}$  can be considered as a \lq\lq limit case" of $\lambda$-soliton for the MCF when $\lambda=0$, because as $\vec{H}_{\Sigma}=\vec{0}$, then it satisfies equation (\ref{geocon}).
\end{remark}
For the inverse mean curvature flow we have the following definition:
\begin{definition}\label{solitonIMCF}
The complete isometric immersion $X:\Sigma^n\to\erre^{n+m}$ is a $C$-soliton for the IMCF with respect $\vec{0} \in \erre^{n+m}$, ($C \in \erre$), if and only if 
\begin{equation}\label{geoconInverse}
\frac{\vec{H}(p)}{\Vert \vec{H}(p)\Vert^2}=-C X^\perp 
\end{equation}
\noindent where $X^\perp$ stands for the normal component of $X$ and $\vec{H}$ is the mean curvature vector of the immersion $X$. 
\end{definition}
\begin{remark}
Note that if we have a complete isometric immersion $X:\Sigma^n\to\erre^{n+m}$ satisfying the geometric condition (\ref{geoconInverse})
and we consider the family of immersions $X_t=e^{Ct} X$, it is straightforward to check that $\{X_t\}_{t=0}^\infty$ satisfies equation (\ref{IMCF}), so $X$ becomes the $0$-slice of the family $\{X_t\}_{t=0}^\infty$ of solutions of equation (\ref{IMCF}).
\end{remark}
\begin{definition}
A $C$-soliton for the IMCF with respect $\vec{0} \in \erre^{n+m}$ is called a self-shrinker if and only if  $C <0$. It is called a self-expander if and only if $C>0$. 
\end{definition}

\begin{remark}\label{minimalIMCF}
A complete and minimal immersion $X:\Sigma^n\to\erre^{n+m}$ cannot be considered as a $C$-soliton for the IMCF with respect $\vec{0} \in \erre^{n+m}$ for any constant $C$ because $X$ cannot satisfy equation (\ref{geoconInverse}).
\end{remark}

%%%%%%%%%%%%%%%%%%%%%
\subsection{Solitons and spherical immersions}\label{solImm}
%%%%%%%%%%%%%%%%%%%%%%%%
Let us consider now a {\em spherical} immersion, namely, an isometric immersion $X:\Sigma^n\to \erre^{n+m}$ such that $X(\Sigma) \subseteq S^{n+m-1}(R)$ for some radius $R>0$. Then, we have the following characterization of self-shrinkers of MCF and self-expanders of IMCF. Assertion $(3)$ concerning solitons for the IMCF was proved in \cite{DLW}, and  it was  proved in \cite{CL1} that {\em closed} $C$-solitons for the IMCF are minimal spherical immersions with velocity $C=\frac{1}{n}$.

Previous to the statement of the characterization, we recall Takahashi's Theorem (see \cite{Tak}), which will be used in our proof:
 \begin{theorem}\label{teorTaka}
If an isometric immersion $\varphi: M^n \to \erre^{n+m}$ of a Riemannian manifold satisfies $\Delta^M \varphi+\lambda \varphi=0$ for some constant $\lambda \neq 0$, then $\lambda >0$ and $\varphi$ realizes a minimal immersion in a sphere $S^{n+m-1}(R)$ with $R=\sqrt{\frac{n}{\lambda}}$.
\end{theorem}
Now, the mentioned result:
\begin{proposition}\label{teorSpherical}
Let $X:\Sigma^n\to\erre^{n+m}$ be a complete spherical immersion.  We have that:
\begin{enumerate}
\item If $X$ is a $\lambda$-soliton for the MCF with respect $\vec{0} \in \erre^{n+m}$, then $\lambda=\frac{n}{R^2}$ and 
$X:\Sigma^n\to S^{n+m-1}(R)$ is a minimal immersion.
\item If $X$ is a $C$-soliton for the IMCF with respect $\vec{0} \in \erre^{n+m}$, then $C=\frac{1}{n}$ and $X:\Sigma^n\to S^{n+m-1}(R)$ is a minimal immersion.
\item Conversely, if $X:\Sigma^n\to S^{n+m-1}(R)$ is a minimal immersion, then $X$ is, simultaneously,  a $\frac{n}{R^2}$-soliton for the MCF with respect $\vec{0} \in \erre^{n+m}$ and a $\frac{1}{n}$-soliton for the IMCF with respect $\vec{0} \in \erre^{n+m}$.
\end{enumerate}
\end{proposition}
\begin{proof}
First of all, note that, as $\Vert X\Vert =R\,\text{on}\,\,\Sigma$ then $ X(q)\perp T_q \Sigma$ for all $q \in \Sigma$. Hence 
$$X^\perp=X \,\text{ and}\,\, X^T=0.$$
To see $(1)$, we have, as $\Sigma$ is a $\lambda$-soliton for the MCF, that 
$$\vec{H}_{\Sigma \subseteq \erre^{n+m}}=-\lambda X^\perp=-\lambda X.$$
\noindent On the other hand, $\lambda \neq 0$ because as $r=R\,\,\,\text{on}\,\,\Sigma$, then, applying Lemma \ref{lemma},
$$0=\Delta^\Sigma r^2=2n-2\lambda R^2,$$
\noindent and hence $\lambda=\frac{n}{R^2} \neq 0$. Therefore, $\Delta^\Sigma X=\vec{H}_{\Sigma \subseteq \erre^{n+m}}=-\frac{n}{R^2} X$.
We apply now Takahashi's Theorem to conclude that $X:\Sigma^n\to S^{n+m-1}(\sqrt{\frac{n}{\lambda}})$ is a minimal immersion.

To see assertion (2), we have that, as $\Sigma$ is a $C$-soliton for the IMCF, that 
$$\frac{\vec{H}_{\Sigma \subseteq \erre^{n+m}}}{\Vert  \vec{H}_{\Sigma \subseteq \erre^{n+m}}  \Vert^2}=-C X^\perp=-C X.$$
\noindent On the other hand, $C \neq 0$ because as $r=R\,\,\,\text{on}\,\,\Sigma$, then, applying Lemma \ref{lemma},
$$0=\Delta^\Sigma r^2=2(n-\frac{1}{C})$$
\noindent and hence $C=\frac{1}{n} \neq 0$. Moreover,
$$\Vert  \frac{\vec{H}_{\Sigma \subseteq \erre^{n+m}}}{\Vert  \vec{H}_{\Sigma \subseteq \erre^{n+m}}  \Vert^2} \Vert=\frac{R}{n}$$
\noindent so $\Vert \vec{H}_{\Sigma \subseteq \erre^{n+m}} \Vert=\frac{n}{R}$, and therefore,
$$\Delta^\Sigma X=\vec{H}_{\Sigma \subseteq \erre^{n+m}}=-C\Vert \vec{H}_{\Sigma \subseteq \erre^{n+m}} \Vert^2X=-\frac{n}{R^2} X.$$
\noindent Again we use  Takahashi's Theorem to conclude that $X:\Sigma^n\to S^{n+m-1}(R)$ is a minimal immersion.

To prove assertion (3), let us suppose that $X:\Sigma^n\to S^{n+m-1}(R)$ is a minimal immersion. Then use the equation, (see \cite{Chavel}):
$$\vec{H}_{\Sigma \subseteq \erre^{n+m}}=\vec{H}_{\Sigma \subseteq S^{n+m-1}(R)} -\frac{n}{R^2}X=-\frac{n}{R^2}X=-\frac{n}{R^2}X^{\perp}$$
\noindent and we have that $\Sigma$ is a $\lambda$-soliton for the MCF with $\lambda=\frac{n}{R^2}$.

On the other hand, $\Vert \vec{H}_{\Sigma \subseteq \erre^{n+m}} \Vert=\frac{n}{R^2}\Vert X\Vert=\frac{n}{R}$, and hence
$$\frac{\vec{H}_{\Sigma \subseteq \erre^{n+m}}}{\Vert  \vec{H}_{\Sigma \subseteq \erre^{n+m}}  \Vert^2}=-\frac{1}{n} X^\perp$$
\noindent and we have that $\Sigma$ is a $C$-soliton for the IMCF, independently of the radius $R$. \end{proof}
%%%%%%%%%%%%%%%%%%%%%
\section{A geometric description of parabolicity of MCF-solitons}\label{geodescMCF}
%%%%%%%%%%%%%%%%%%%%%%%
%%%%%%%%%%%%%%%
%\subsection{Solitons for MCF}\
%%%%%%%%%%%%

%%%%%%%%%%%%%%%%%%%%%%%%%%%
\subsection{Geometric necessary conditions for parabolicity}\label{necondMCF}
%%%%%%%%%%%%%%%%%%%%%%%%%%%

We start proving that parabolic solitons for MCF with dimension strictly greater than $2$ are self-shrinkers.

\begin{theorem}\label{teorNecMCF}
Let $X:\Sigma^n\to\erre^{n+m}$ be a complete and parabolic $\lambda$-soliton for the MCF with respect $\vec{0} \in \erre^{n+m}$, with $n>2$.  Then $X$ is  a self-shrinker ($\lambda>0$) for the MCF.
\end{theorem}
\begin{proof}
To prove the theorem we are going to apply Theorem \ref{teo0} with a family of bounded functions depending on $\epsilon >0$ and constructed using the distance function. For any $\epsilon>0$, let us consider
the function $f_1^\epsilon: \erre^*_+ \rightarrow (-\infty, \frac{1}{\epsilon})$ defined as
$$f_1^{\epsilon}(s)=\frac{1}{\epsilon}(1-\frac{1}{s^{\epsilon}})$$
The function $f_1^{\epsilon}$ is smooth in $\erre^*_{+}$ and strictly increasing in $\erre^*_{+}$, so it is a bijection among $\erre^*_{+}$ and its image ${\rm Im} f_1^{\epsilon}$. Moreover, as $\lim_{r \to 0^+} f_1^{\epsilon}(r)= -\infty$ and $\lim_{r \to \infty} f_1^{\epsilon}(r)= \frac{1}{\epsilon}$, then $\sup_{\erre_{+}} f_1^{\epsilon} \leq \frac{1}{\epsilon} < \infty$. 

We are going to divide the rest of the proof in two cases. First, we shall consider a soliton $\Sigma$ such that $\vec{0} \notin X(\Sigma)$. In this case, $r^{-1}(0) =\emptyset$, and we define the functions,
\begin{equation}\label{defu1}
 u_1^\epsilon:\Sigma\to\mathbb{R},\quad x\to u_1^{\epsilon}(x):=f_1^{\epsilon}(r(x)).
 \end{equation}
 
 \noindent We have that $\sup_\Sigma u_1^{\epsilon}=u_1^{\epsilon^*} \leq \frac{1}{\epsilon}<\infty$, and, as $\vec{0} \notin \Sigma$, then $r^{-1}(0) =\emptyset$, and these functions are smooth in $\Sigma$. Then we can apply to them directly Theorem \ref{teo0} in the following way:

If, for some $\epsilon >0$, the function $u_1^{\epsilon}$ is constant, then it is straightforward to check that {\em all}  functions $u_1^{\epsilon}$ are constant and, moreover, $r\vert_{\Sigma}=R$, so $X(\Sigma) \subseteq S^{n+m-1}(R)$, namely, $X$ is a spherical immersion and hence, we apply Proposition \ref{teorSpherical} to get the conclusion (1), (for all $n \geq 1$).

Alternatively, let us suppose that the test functions $u_1^\epsilon$ are nonconstant on $\Sigma^n$. Given $\epsilon >0$, since $\displaystyle\sup_{\Sigma}u<\infty$ and $\Sigma$ is parabolic, we know  by using Theorem \ref{teo0} that there exists a sequence $\{x_k\}\subset \Sigma$, (depending on $\epsilon$), such that
$$
\triangle^\Sigma u_1^{\epsilon}(x_k)<0
$$
Moreover, by equation (\ref{eq2.3})
$$\begin{aligned}
0>\triangle u_1^\epsilon(x_k)=&-\frac{2+\epsilon}{r^{4+\epsilon}(x_k)}\Vert X^T(x_k)\Vert^2+\frac{1}{r^{2+\epsilon}(x_k)}\left(n+\langle H, X\rangle\right)\\
\geq & -\frac{2+\epsilon}{r^{4+\epsilon}(x_k)}\Vert X(x_k)\Vert^2+\frac{1}{r^{2+\epsilon}(x_k)}\left(n+\langle H, X\rangle\right)\\
=& \frac{-2-\epsilon+n-\lambda\Vert X^\perp(x_k)\Vert^2}{r^{2+\epsilon}(x_k)}
\end{aligned}
$$
where we have used that $\displaystyle \langle H, X\rangle =-\lambda\Vert X^\perp\Vert^2$ because $X:\Sigma\to \erre^{n+m}$ is a $\lambda$-soliton for the MCF with respect $\vec{0} \in \erre^{n+m}$.Therefore, for any $\epsilon>0$, and for its associated sequence $\{x_k\}\subset \Sigma$,
$$
\lambda\Vert X^\perp(x_k)\Vert^2>n-2-\epsilon
$$
\noindent  Then, if $n>2$ there exists $\epsilon_0$ such that $n-2-\epsilon_0 >0$ and we have that
$$
\lambda\Vert X^\perp(x^{\epsilon_{0}}_k)\Vert^2>n-2-\epsilon_0>0
$$
 \noindent so we conclude that $\lambda>0$ and we have proved the theorem.

In the second case to consider, we assume that $\vec{0} \in \Sigma$, namely, that $X^{-1}(\vec{0})\neq \emptyset$. Then, $r^{-1}(0) \neq \emptyset$, so $u_1^{\epsilon}$  is not smooth in $r^{-1}(0) \subseteq \Sigma$. We are going to modify $u_1^{\epsilon}$ to get $u_2^{\epsilon} \in \mathcal{C}^\infty(\Sigma)$ and we shall use the same argument  than before on these modified functions with some care. This modification is given by the following

\begin{lemma}\label{lemodif}Let $X:\Sigma\to \mathbb{R}^{n+m}$ be an isometric immersion. Suppose that $X^{-1}(\vec{0})\neq \emptyset$. Then, given $\epsilon >0$ and the function $u_1^\epsilon$ defined in equation (\ref{defu1}), there exist a smooth function $u^{\epsilon}_2: \Sigma \rightarrow \erre$ and a positive real number $x_0>0$ such that
\begin{enumerate}
\item The function $u_2^\epsilon$ satisfies that,
$$
u_2^\epsilon=\left\{\begin{array}{lcc}u_1^\epsilon&{\rm on}&\Sigma\setminus D_{\frac{x_0}{2}}\\
f_1^\epsilon(\frac{x_0}{4})&{\rm on}&D_{\frac{x_0}{4}}.
\end{array}\right.
$$ 
\item The function $u_2$ is not constant on $\Sigma$, and $\sup_{\Sigma}u_2^\epsilon > \sup_{D_{\frac{x_0}{2}}}u_2^\epsilon$. 

Therefore,
$$\sup_{\Sigma}u_2^\epsilon \leq u_1^{\epsilon^*}<\infty.$$
%\item Let $\delta_0>0$ such that $\sup_{\Sigma}u_2^\epsilon - \delta_0 > \sup_{D_{\frac{x_0}{2}}}u_2^\epsilon$. Then, there exists $p_{\delta_{0}} \in \Sigma$ such that $u_2^\epsilon(p_{\delta_{0}})>\sup_{\Sigma}u_2^\epsilon - \delta_0 > \sup_{D_{\frac{x_0}{2}}}u_2^\epsilon$.}
\end{enumerate}
\end{lemma}
\begin{proof}
To prove the Lemma, and given the function $u_1^\epsilon$ defined in equation (\ref{defu1}), let us consider an extrinsic ball $D_\rho(\vec{0})\subseteq \Sigma$ such that $\Sigma\setminus D_\rho(\vec{0}) \neq \emptyset$. We have that $u_1^\epsilon \in C^\infty(\Sigma\setminus D_\rho(\vec{0}))$. Let us fix $0 <x_0 <\rho$ and $0<\delta_0<f_1^\epsilon(x_0)-f_1^\epsilon(\frac{x_0}{2})$, and let us  define the function $g^\epsilon: (-\infty, \frac{x_0}{4}] \cup [\frac{x_0}{3}, \frac{x_0}{2}) \rightarrow ( f^\epsilon_1(\frac{x_0}{4})-\delta_0,  f^\epsilon_1(\frac{x_0}{2})+\delta_0)$ as
\begin{equation}
g^\epsilon(s):=\left\{\begin{array}{lcc}
         f^\epsilon_1(\frac{x_0}{4})&{\rm for}& s \leq \frac{x_0}{4},\\
        f^\epsilon_1(s)&{\rm for}& s \geq \frac{x_0}{3}.\\
        \end{array}
        \right.
\end{equation}
The set $A:=(-\infty, \frac{x_0}{4}] \cup [\frac{x_0}{3}, \frac{x_0}{2})$ is closed in $N:=(-\infty,  \frac{x_0}{2})$, and if we denote as $M:=( f^\epsilon_1(\frac{x_0}{4})-\delta_0,  f^\epsilon_1(\frac{x_0}{2})+\delta_0)$, then $g^\epsilon \in \mathcal{C}^\infty(A,M)$. Moreover, there is a continuous extension of $g^\epsilon$ to $N$, given by
\begin{equation}
h^\epsilon(s):=\left\{\begin{array}{lcc}
         f^\epsilon_1(\frac{x_0}{4})&{\rm for}& s \leq \frac{x_0}{4},\\
         (s-\frac{x_0}{4})\frac{f^\epsilon_1(\frac{x_0}{3})-f^\epsilon_1(\frac{x_0}{4})}{\frac{x_0}{3}-\frac{x_0}{4}}+ f^\epsilon_1(\frac{x_0}{4})&{\rm for}& \frac{x_0}{4}\leq s \leq \frac{x_0}{3},\\
        f^\epsilon_1(s) &{\rm for}&\frac{x_0}{3}\leq s \leq \frac{x_0}{2}.\\
        \end{array}
        \right.
\end{equation}
Then, applying the Extension Lemma for smooth maps, (see \cite{Lee}), there exists an smooth extension $\bar{h^\epsilon}:  N \rightarrow M$ of $g^\epsilon$, i.e. $\bar{h^\epsilon}\vert_A=g^\epsilon$. This function $\bar{h^\epsilon}$ can be trivially extended smoothly to all the real line defining $f_2^\epsilon:(-\infty, \infty) \rightarrow (f_1^\epsilon(\frac{x_0}{4})-\delta_0, \frac{1}{\epsilon})$ as
\begin{equation}
f_2^\epsilon(s):=\left\{\begin{array}{lcc}
        \bar{h^\epsilon}(s)&{\rm for}&s<\frac{x_0}{2},\\
        f^\epsilon_1(s)&{\rm for}&s \geq \frac{x_0}{2}.\\
        \end{array}
        \right.
\end{equation}
\noindent because $\bar{h^\epsilon}(s)=g^\epsilon(s)=f^\epsilon_1(s)$ for any $s > \frac{x_0}{3}$, and hence, 
$f_2^\epsilon=f_1^\epsilon=\bar{h^\epsilon}$ in the open set $(\frac{x_0}{3},\frac{x_0}{2})$.

Now, let us define, for each $\epsilon >0$, the function  $u_2^\epsilon: \Sigma \rightarrow \erre$ as $u_2^\epsilon(p):= f_2^\epsilon (r(p))$. Then, $u_2^\epsilon \in \mathcal{C}^\infty(\Sigma)$. Observe that this $u_2^\epsilon$ satisfies the statement (1) of the lemma. 

To prove statement (2) of the lemma note that since $X^{-1}(\vec{0})\neq \emptyset$ there exist at least one point $p\in \Sigma$ such that $p\in D_{\frac{x_0}{4}}$, (on the contrary, $X(\Sigma) \subseteq \erre^{n+m} \setminus B^{n+m}_{\frac{x_0}{4}}(\vec{0})$, so $X^{-1}(\vec{0})= \emptyset$).
Then $u_2(p)=f_1^\epsilon(\frac{x_0}{4})$. On the other hand, since $\Sigma\setminus D_{x_0}\neq \emptyset$, then there exist at least one $q\in \Sigma\setminus D_{x_0}$. Then, as $f_1^\epsilon$ is strictly increasing,
$$u_2^\epsilon(q)=u_1^\epsilon(q)=f_1^\epsilon(r(q))\geq f_1^\epsilon(x_0)>f_1^\epsilon(\frac{x_0}{4})=u_2^\epsilon(p).$$
Hence, $u_2^\epsilon$ is not constant on $\Sigma$. 
Let us observe now that, as $\delta_0 < f_1^{\epsilon}(x_0)-f_1^\epsilon(\frac{x_0}{2})$, and $f_2^\epsilon(s) =\bar{h^\epsilon}(s)\,\,\forall s<\frac{x_0}{2}$, then we have
$$
\sup_{D_{\frac{x_0}{2}}}u_2^\epsilon\leq f_1^\epsilon(\frac{x_0}{2})+\delta_0<f_1^\epsilon(x_0)
$$
But again since $\emptyset \neq \Sigma\setminus D_{\rho}\subseteq \Sigma\setminus D_{x_0}$, there exists $q\in \Sigma\setminus D_{x_0}$ with $u_2^\epsilon(q)\geq f_1^\epsilon(x_0)$. Then,
$$
\sup_\Sigma u_2^\epsilon>\sup_{D_{\frac{x_0}{2}}}u_2^\epsilon.
$$
\noindent Now, let us suppose that $\sup_{D_{\frac{x_0}{2}}}u_2^\epsilon > \sup_{\Sigma}u_1^\epsilon $. Then, as $\sup_{\Sigma}u_1^\epsilon \geq \sup_{\Sigma\setminus D_{\frac{x_0}{2}}}u_1^\epsilon= \sup_{\Sigma \setminus D_{\frac{x_0}{2}}}u_2^\epsilon$, we obtain $\sup_{D_{\frac{x_0}{2}}}u_2^\epsilon > \sup_{\Sigma\setminus D_{\frac{x_0}{2}}}u_2^\epsilon$ and therefore, $\sup_{D_{\frac{x_0}{2}}}u_2^\epsilon \geq \sup_\Sigma u_2^\epsilon$, which is a contradiction. Hence, $\sup_{D_{\frac{x_0}{2}}}u_2^\epsilon \leq \sup_{\Sigma}u_1^\epsilon $ and therefore, as we know that $\sup_{\Sigma\setminus D_{\frac{x_0}{2}}}u_2^\epsilon= \sup_{\Sigma \setminus D_{\frac{x_0}{2}}}u_1^\epsilon \leq \sup_{\Sigma}u_1^\epsilon$, then  $\sup_{\Sigma}u_2^\epsilon \leq \sup_{\Sigma}u_1^\epsilon$.  

\end{proof}
We can finish now the proof of the theorem by using as a test function in Theorem \ref{teo0} the smooth function $u_2^\epsilon$ given by Lemma \ref{lemodif}.  For any $\epsilon >0$, since $\sup_{\Sigma}u_2^\epsilon<\infty$ and $\Sigma$ is parabolic, we know  by using Theorem \ref{teo0} that there exists a sequence $\{x_k\}\subset \Sigma$ such that $u_2^\epsilon(x_k) \geq u_2^{\epsilon^*}-\frac{1}{k}$ and 
$$
\triangle^\Sigma u_2^{\epsilon}(x_k)<0
$$
Then, as $\sup_\Sigma u_2^\epsilon>\sup_{D_{\frac{x_0}{2}}}u_2^\epsilon$, there exists $\delta_1>0$ such that $ \sup_\Sigma u_2^\epsilon-\delta_1 >\sup_{D_{\frac{x_0}{2}}}u_2^\epsilon$. Given the sequence $\{x_k\}\subset \Sigma$, let us consider the numbers $k$ such that $\frac{1}{k}  <\delta_1$. Then
$$u_2^\epsilon(x_k) >  sup_\Sigma u_2^\epsilon-\frac{1}{k} >\sup_\Sigma u_2^\epsilon-\delta_1 >\sup_{D_{\frac{x_0}{2}}}u_2^\epsilon$$
\noindent so $x_k$  belongs to $\Sigma\setminus D_{\frac{x_0}{2}}$ for $k$ large enough and  we have
$$\begin{aligned}
0>\triangle u_2^\epsilon(x_k)=\triangle u_1^\epsilon(x_k)&=-\frac{2+\epsilon}{r^{4+\epsilon}(x_k)}\Vert X^T(x_k)\Vert^2+\frac{1}{r^{2+\epsilon}(x_k)}\left(n+\langle H, X\rangle\right)\\
\geq & -\frac{2+\epsilon}{r^{4+\epsilon}(x_k)}\Vert X(x_k)\Vert^2+\frac{1}{r^{2+\epsilon}(x_k)}\left(n+\langle H, X\rangle\right)\\
=& \frac{-2-\epsilon+n-\lambda\Vert X^\perp(x_k)\Vert^2}{r^{2+\epsilon}(x_k)}
\end{aligned}
$$
\noindent and we follow the argument as in the first case.\end{proof}

As a first Corollary of Theorem \ref{teorNecMCF} we have the following result, which extends one of the results in \cite{MP}, (namely, that complete, non-compact and minimal immersions $X:\Sigma^n \to \erre^{n+m}$ with $n>2$ are non-parabolic), to self-expanders for the MCF, not necessarily proper.

\begin{corollary}\label{cor7}
 Let $X:\Sigma^n \to \erre^{n+m}$ be a complete $\lambda$-soliton for the MCF, with $\lambda \leq 0$ and $n >2$. Then $\Sigma$ is non-parabolic.
\end{corollary}

\begin{remark}
We must note at this point that the converse in Theorem \ref{teorNecMCF} is not true in general: 

\begin{enumerate}
\item When $n=1$, then  complete$\lambda$-solitons for the MCF are parabolic for all $\lambda$.
\item  When $n=2$, we have examples of  complete parabolic and non-parabolic minimal surfaces, (e.g., the catenoid is parabolic and the doubly-periodic Scherk's surface is non parabolic). On the other hand, the spheres $S^{1}(R)$ and the cylinders $S^{1}(R)\times \erre$ are parabolic $2$-dimensional MCF-self-shrinkers. 
%We cannot say anything about the existence of non-parabolic $2$-dimensional MCF-self shrinkers or about the existence of non-parabolic $2$-dimensional MCF- self expanders.
\item in Section \ref{exam} we will show examples of $2$-dimensional parabolic self-expanders and parabolic and non-parabolic self-shrinkers with $n>3$, ($S^{n-2}(R)\times \erre^2$ and $S^{n-k}(R)\times \erre^k$ with $k>2$ respectively). 

\end{enumerate}

\end{remark}

Concerning the behavior of two dimensional properly-immersed and parabolic self-expanders for the MCF, we have the following result: 

\begin{corollary}\label{cor8} Let $X:\Sigma^2\to \erre^{2+m}$ be an immersed, complete and parabolic self-expander for the MCF. Then 
$$\displaystyle\inf_{\Sigma^n}\Vert H\Vert =0.$$
Moreover, if $X$ is proper then, for any $R>0$ and any connected and unbounded component $V$ of  $\left\{ p \in \Sigma\,: \, \Vert X(p)\Vert>R\right\}$, we have
 $$
 \inf_V\Vert \vec{H}\Vert=0.
 $$
\end{corollary}
\begin{proof}
To prove the first assertion, we are going to apply Theorem \ref{teo0} as in Theorem \ref{teorNecMCF} with the same family of bounded functions $\{u_1^\epsilon\}_{\epsilon >0}$ depending on $\epsilon >0$ and constructed using the distance function.

Then, if we assume that $\vec{0} \notin X(\Sigma)$, we have, for all $\epsilon >0$ and each function $u_1^\epsilon \in \mathcal{C}^\infty(\Sigma)$ , a sequence $\{x^\epsilon_k\}\subset \Sigma$, (depending on $\epsilon$), such that $\triangle^\Sigma u_2^{\epsilon}(x_k)<0$ and therefore
$$
\lambda\Vert X^\perp(x^\epsilon_k)\Vert^2> -\epsilon
$$
\noindent so
$$
\Vert X^\perp(x^\epsilon_k)\Vert^2 < \frac{-\epsilon}{\lambda}.
$$
Since $\Vert \vec{H}\Vert^2=\lambda^2\Vert X^\perp\Vert^2$, we have, for each sequence $\{x^\epsilon_k\}\subset \Sigma$, depending on $\epsilon$
$$
\Vert \vec{H}(x^\epsilon_k)\Vert^2 < -\epsilon\lambda,
$$
\noindent which implies that, for all $\epsilon >0$,

$$
\displaystyle\inf_{\Sigma^n}\Vert \vec{H}\Vert^2 \leq -\epsilon\lambda,
$$
\noindent and hence 
$$
\displaystyle\inf_{\Sigma^n}\Vert \vec{H}\Vert^2 =0.
$$

On the other hand, if  we assume that $\vec{0} \in X(\Sigma)$, we argue as in the proof of Theorem \ref{teorNecMCF}, modifying $u_1^{\epsilon}$ to obtain a new function $u_2^{\epsilon} \in \mathcal{C}^\infty(\Sigma)$ which satisfies Lemma \ref{lemodif}. As we have seen before, these new functions cannot be constant, so we apply Lemma \ref{lemma} and Theorem \ref{teo0} again, obtaining, for each $\epsilon >0$, and each function $u_2^\epsilon \in \mathcal{C}^\infty(\Sigma)$ , a sequence $\{x_k\}\subset \Sigma$, (depending on $\epsilon$),
$$
\lambda\Vert X^\perp(x_k)\Vert^2> -\epsilon
$$
\noindent Now the proof follows as above.

Finally, to prove second assertion, for any connected and unbounded component $V$   of $\Sigma\setminus D_R$ we define the following  function
\begin{equation*}
F^\epsilon_V(x):=\left\{\begin{aligned}
&f_1^\epsilon(R)\quad \text{ if } x\in D_R\\
&u_1^\epsilon(x)\quad \text{ if } x\in \left(\Sigma\setminus D_{2R}\right)\cap V\\
& f_1^\epsilon(2R)\quad \text{ if } x\in\left(\Sigma\setminus D_{2R}\right)\setminus\left( \left(\Sigma\setminus D_{2R}\right)\cap V\right)
\end{aligned}\right.
\end{equation*}
Observe that $F_V^\epsilon$ is a smooth function defined on $D_R\cup \left(\Sigma\setminus D_{2R}\right)$ and has a continuous extension on $D_{2R}\setminus D_R$. Then, by using  similar arguments  as the used in the proof of Lemma  \ref{lemodif} there exists an smooth extension $\overline F_V^\epsilon:\Sigma\to \mathbb{R}$. Since $\overline{F}^\epsilon_V$ is bounded and is non-constant, by theorem \ref{teo0} there exists a sequence $\{x_k\}_{k\in \mathbb{N}}$ such that
$$
\overline{F}_V^\epsilon(x_k)>\sup_{\Sigma}\overline{F}^\epsilon_V-\frac{1}{k},\quad \triangle \overline{F}_V^\epsilon(x_k)<0.
$$  
This implies that $\{x_k\}$ belongs to $V$ for $k$ large enough, and hence $\overline{F}_V^\epsilon(x_k)=u_1^\epsilon(x_k)$. Furthermore,
$$
\triangle \overline{F}_V^\epsilon(x_k)=\triangle u_1^\epsilon(x_k)<0
$$
Then,
$$
\inf_V\Vert \vec{H}\Vert^2\leq \Vert \vec{H}(x_k)\Vert^2\leq -\epsilon \lambda.
$$
Finally the corollary follows letting again $\epsilon$ tend to $0$.
\end{proof}
\begin{remark}
As a consequence of Corollary \ref{cor8}, if $\Sigma^2$ is a proper self-expander for the MCF and $\Vert \vec{H}_\Sigma\Vert >C$ out of a compact set in $\Sigma^2$, then $\Sigma^2$ is non parabolic
\end{remark}

%%%%%%%%%%%%%%%%%%%%%%%%%%%
\subsection{Geometric sufficient conditions for parabolicity}\label{sufcondMCF}
%%%%%%%%%%%%%%%%%%%%%%%%%%%

We are going to study now sufficient conditions for parabolicity of properly immersed solitons for the MCF. In the paper \cite{Rim}, M. Rimoldi has shown the following theorem, which shows that proper self-shrinkers for the MCF with mean curvature bounded from below  exhibits the opposite behavior than we have pointed out in Remark above for proper self-expanders satisfying the same property.  We give the proof here for completeness:

\begin{theorem}\label{teo2a} 
Let $X:\Sigma^n\to \erre^{n+m}$ be a complete and non-compact properly immersed $\lambda$-self-shrinker for the MCF with respect $\vec{0} \in \erre^{n+m}$. If $\Vert \vec{H}_\Sigma\Vert \geq \sqrt{n \lambda}$ outside a compact set, then $\Sigma$ is a parabolic manifold. In particular, if $\Vert \vec{H}_\Sigma\Vert \to \infty$ when $x \to \infty$, then $\Sigma$ is parabolic.
 \end{theorem}
\begin{proof}

Given $r^2=\Vert X(p) \Vert^2$, we have that 
$$\Delta^\Sigma r^2= 2(n-\frac{1}{\lambda}\Vert H_\Sigma\Vert^2)\leq 0$$ 
\noindent As $\vec{H}_\Sigma \to \infty$ when $x \to \infty$ and $X$ is proper, then $\Delta^\Sigma r^2 \leq 0$ outside a compact set. Then, apply Theorem \ref{k} to get the conclusion.\end{proof}

\begin{remark}
As a consequence of Corollary \ref{cor8} and Theorem \ref{teo2a}, we can conclude that, if they exists, {\em all}  complete and non-compact non-parabolic $n$-dimensional self-shrinkers for the MCF, such that $\Vert \vec{H}_\Sigma\Vert \geq \sqrt{n \lambda}$ outside a compact set, {\em are not} properly immersed. 

Respectively, if they exists, {\em all}  complete and non-compact parabolic  $n$-dimensional self-expanders for the MCF ($n >2$), such that $\Vert \vec{H}_\Sigma\Vert \geq C$  outside a compact set, being $C$ any positive constant, {\em are not} properly immersed. 

These affirmations come from the fact that, in case  $X:\Sigma^n\to \erre^{n+m}$ is a complete and non-compact properly immersed self-shrinker for the MCF, (resp. self-expander), satisfying $\Vert \vec{H}_\Sigma\Vert \geq \sqrt{n \lambda}$ outside a compact set, (resp. $\Vert \vec{H}_\Sigma\Vert \geq C$  outside a compact set, being $C$ any positive constant), then $\Sigma$ must be parabolic, (resp., non-parabolic).\end{remark}

To prove our last sufficient condition of parabolicity for properly immersed solitons for MCF, we shall prove first the following result, which shows that, in some sense, (see affirmation (3) in the statement of the Theorem), MCF-self-shrinkers behaves in a similar way than minimal immersions in the sphere even when they are not minimal immersions.

\begin{theorem}\label{teorSuffMCF}
 Let $X:\Sigma\to \erre^{n+m}$ be a complete properly immersed $\lambda$-self-shrinker for the MCF with respect $\vec{0} \in \erre^{n+m}$,  then
 \begin{enumerate}
  \item $\displaystyle\int_\Sigma e^{-\frac{\lambda }{2}r^2(p)}dV(p)<\infty$.
  \item $\displaystyle\int_\Sigma r^2(p)e^{-\frac{\lambda }{2}r^2(p)}dV(p)<\infty$.
  \item $\displaystyle\lambda\int_\Sigma r^2(p)e^{-\frac{\lambda }{2}r^2(p)}dV(p)=n\int_\Sigma e^{-\frac{\lambda }{2}r^2(p)}dV(p)$
 \end{enumerate}
where here,
$
r(p):=\Vert X(p)\Vert
$
and $dV$ stands for the Riemannian volume density of $\Sigma$.
\end{theorem}
\begin{proof}
 If $X$ is spherical, affirmations (1) and (2) are obvious. Moreover, in this case affirmation (3) follows from Proposition \ref{teorSpherical} because $\lambda=\frac{n}{R^2}$. On the other hand, if $X$ is not a spherical immersion, the statement (1) of the theorem is proved in \cite{CZ}. 
 
 To prove (2) and (3), since  the immersion is proper, the extrinsic ball 
 $
 D_R
 $
 , \emph{i.e.}, $\{x\in \Sigma\,:\, \Vert X(x)\Vert<R\}$
 is a precompact set of $\Sigma$ and its boundary
 $$
 \partial D_R =\left\{x\in \Sigma\,:\, \Vert X(x)\Vert=R\right\}
$$
by the Sard's theorem is a smooth submanifold of $\Sigma$ for almost every $R$ with unit normal vector field $\frac{\nabla r}{\Vert \nabla r\Vert}$. 
then by applying the divergence theorem on $D_R$ to the vector field $e^{-\frac{\lambda}{2}r^2}\nabla r^2$, we obtain
\begin{equation}\label{eqdiV}
\begin{aligned}
 \int_{D_R}{\rm div}\left(e^{-\frac{\lambda}{2}r^2}\nabla r^2\right)dV=&\int_{\partial D_R}e^{-\frac{\lambda}{2}r^2}\langle \nabla r^2,\frac{\nabla r}{\Vert \nabla r\Vert}\rangle dA\\
 =&2R e^{-\frac{\lambda}{2}R^2}\int_{\partial D_R} \Vert \nabla r\Vert dA.
\end{aligned}
\end{equation}
 But, taking into account that
 \begin{equation}\label{eq3.6.2}
 \begin{aligned} 
 {\rm div}\left(e^{-\frac{\lambda}{2}r^2}\nabla r^2\right)=&\langle \nabla e^{-\frac{\lambda}{2}r^2},\, \nabla r^2\rangle +e^{-\frac{\lambda}{2}r^2}\triangle r^2\\
 =&-2\lambda e^{-\frac{\lambda}{2}r^2}r^2\Vert \nabla r \Vert^2 +e^{-\frac{\lambda}{2}r^2}\left(2n-2\lambda\Vert X^\perp\Vert^2 \right)\\
 =&-2\lambda e^{-\frac{\lambda}{2}r^2}\Vert X^T \Vert^2 +e^{-\frac{\lambda}{2}r^2}\left(2n-2\lambda\Vert X^\perp\Vert^2 \right)\\
 =&2e^{-\frac{\lambda}{2}r^2}\left(n-\lambda r^2\right),
 \end{aligned}
 \end{equation}
equation (\ref{eqdiV}) can be written as
\begin{equation}\label{eq3.1}
\int_{D_R}e^{-\frac{\lambda}{2}r^2}\left(n-\lambda r^2\right)dV=R e^{-\frac{\lambda}{2}R^2}\int_{\partial D_R} \Vert \nabla r\Vert dA\geq 0.
 \end{equation}
Consequently,
$$
\lambda \int_{D_R} r^2e^{-\frac{\lambda}{2}r^2}dV\leq n\int_{D_R}e^{-\frac{\lambda}{2}r^2}dV\leq n\int_{\Sigma}e^{-\frac{\lambda}{2}r^2}dV.
$$
But then,
$$
\lambda \int_{\Sigma} r^2e^{-\frac{\lambda}{2}r^2}dV=\lim_{R\to \infty}\lambda \int_{D_R} r^2e^{-\frac{\lambda}{2}r^2}dV\leq  n\int_{\Sigma}e^{-\frac{\lambda}{2}r^2}dV
$$
and the statement (2) of the theorem is proved. To prove statement (3) of the theorem, observe that
$$
2n\, {\rm vol}(D_R)\geq \int_{D_R}\triangle r^2 dV=2R\int_{\partial D_R}\Vert \nabla r\Vert dA .
$$
Then by the equality  (\ref{eq3.1}),
$$
0\leq \int_{D_R}e^{-\frac{\lambda}{2}r^2}\left(n-\lambda r^2\right)dV\leq n e^{-\frac{\lambda}{2}R^2} \, {\rm vol}(D_R)\leq n C e^{-\frac{\lambda}{2}R^2} R^n 
$$
where we have applied that for \cite{CZ} since $X$ is proper $\Sigma$ has at most Euclidean volume growth. Finally the theorem is proved by taking the limit $R\to \infty$. \end{proof}
The above Theorem implies that proper self-shrinkers have finite weighted volume when we consider the density $r^2e^{-\frac{\lambda}{2}r^2}$, this property can be used to obtain a sufficient condition for parabolicity. We  shall need the following
\begin{definition}
Let $X:\Sigma\to \erre^{n+m}$ be a proper isometric immersion. 
Let us define the function $\Psi_\Sigma: \erre^+ \rightarrow \erre^+$ as
$$
\Psi_\Sigma(R):=\underset{\{p\in\Sigma\,:\, \Vert X(p)\Vert>R\}}{\int} r^2(p)e^{-\frac{\lambda }{2}r^2(p)}dV(p)
$$
\end{definition}
Because Theorem \ref{teorSuffMCF}, if $X:\Sigma\to \erre^{n+m}$ is a proper  self-shrinker by MCF, then
$$
\lim_{R\to \infty}\Psi_\Sigma(R)=0.
$$
The rhythm of this decay implies in some cases consequences for the parabolicity of $\Sigma$ as the following theorem shows 

\begin{theorem}\label{teorSuffMCF2}
 Let $X:\Sigma\to \erre^{n+m}$ be a complete properly immersed $\lambda$-self-shrinker for the MCF with respect $\vec{0} \in \erre^{n+m}$. Suppose that
 $$
 \int^\infty\frac{te^{-\frac{\lambda}{2}t^2}}{\Psi_\Sigma(t)}dt=\infty.
 $$
 Then, $\Sigma$ is parabolic.
\end{theorem}
\begin{proof}
 We are going to apply Theorem \ref{teo2.2} to the function $u(x)=r(x)=\Vert X(x)\Vert$. By the equality in (\ref{eq3.1}),
 $$
 \begin{aligned}
 \int_{\{x\,:\, r(x)=t\}}\Vert \nabla r\Vert dA=&\int_{\partial D_t}\Vert \nabla r\Vert dA=\frac{e^{\frac{\lambda}{2}t^2}}{t}\int_{D_t}e^{-\frac{\lambda}{2}r^2}\left(n-\lambda r^2\right)dV\\
 =& \frac{e^{\frac{\lambda}{2}t^2}}{t}\left(\int_{\Sigma}e^{-\frac{\lambda}{2}r^2}\left(n-\lambda r^2\right)dV-\int_{\Sigma\setminus D_t}e^{-\frac{\lambda}{2}r^2}\left(n-\lambda r^2\right)dV \right)\\
 =&\frac{e^{\frac{\lambda}{2}t^2}}{t}\int_{\Sigma\setminus D_t}e^{-\frac{\lambda}{2}r^2}\left(\lambda r^2-n\right)dV \\
 \leq &\lambda\frac{e^{\frac{\lambda}{2}t^2}}{t}\int_{\Sigma\setminus D_t}r^2e^{-\frac{\lambda}{2}r^2} dV=\lambda\frac{e^{\frac{\lambda}{2}t^2}}{t}\Psi_\Sigma(t)
 \end{aligned}
 $$
 By using inequality (\ref{eq2.1}) and Theorem \ref{teo2.2} with $K=D_\rho$ and $G=D_R$ with $R>\rho>0$ we obtain,
 $$
 \begin{aligned}
  {\rm cap}(D_\rho,\Sigma)\leq {\rm cap}(D_\rho,D_R)\leq &\left(\int_\rho^R\frac{dt}{\int_{\partial D_t}\Vert \nabla r\Vert dA}\right)^{-1}\\
\leq &  \left(\int_\rho^R\frac{te^{-\frac{\lambda}{2}t^2}}{\lambda\Psi_\Sigma(t)}dt\right)^{-1}
 \end{aligned}
 $$
 Finally the theorem is proved letting $R$ tend to $\infty$\end{proof}
%%%%%%%%%%%%
\section{A geometric description of parabolicity of IMCF-solitons}\
%%%%%%%%%%%%%
As in the previous section, we start with a necessary condition for parabolicity:
\begin{theorem}\label{teorNecIMCF}
 Let $X:\Sigma^n\to \erre^{n+m}$ be a complete soliton for the IMCF, ($n \geq 1$). Then, if $\Sigma$ is a parabolic manifold, then $X$ is a self-expander for the IMCF and
 $$
\frac{1}{n} \leq C\leq \frac{1}{n-2}\cdot
 $$
 \noindent Moreover, if  $C=\frac{1}{n}$, then $X:\Sigma^n\to S^{n+m-1}(R)$ is minimal  for some radius $R>0$.
\end{theorem}
\begin{proof} 
 Given $\epsilon >0$, let us consider the test function $u_{\epsilon}(p):=\frac{1}{\epsilon}(1-\frac{1}{r^\epsilon(p)})$. We have that $\sup_{\Sigma} u_{\epsilon} < \infty$ and $u_\epsilon \in C^2(\Sigma)$ because $\vec{0} \notin X(\Sigma$).  If any of these functions is constant for some $\epsilon >0$, then all are constant  and hence $r=R$ is constant on $\Sigma$. Then, $X:\Sigma^n\to \erre^{n+m}$ is a complete $C$-soliton for the IMCF such that $x(\Sigma)\subseteq S^{n+m-1}(R)$. Hence, applying Proposition \ref{teorSpherical}, $C=\frac{1}{n}$ and $\Sigma$ is minimal in the sphere $S^{n+m-1}(R)$.

Alternatively, let us suppose that the test functions $u_\epsilon$ are nonconstant on $\Sigma$ for all $\epsilon >0$.
Since $\sup_{\Sigma}u_\epsilon<\infty$ and $\Sigma$ is parabolic, we know  by using Theorem \ref{teo0} that there exists a sequence $\{x_k\}\subset \Sigma$ such that
$$
\Delta u_{\epsilon}(x_k)<0
$$
Moreover, by equation (\ref{eq2.3})
$$\begin{aligned}
0>\triangle u_\epsilon(x_k)=&-\frac{2+\epsilon}{r^{4+\epsilon}(x_k)}\Vert X^T\Vert^2+\frac{1}{r^{2+\epsilon}(x_k)}\left(n+\langle H, X\rangle\right)\\
\geq & -\frac{2+\epsilon}{r^{4+\epsilon}(x_k)}\Vert X\Vert^2+\frac{1}{r^{2+\epsilon}(x_k)}\left(n+\langle H, X\rangle\right)\\
=& \frac{-2-\epsilon+n-\frac{1}{C}}{r^{2+\epsilon}(x_k)}
\end{aligned}
$$
where we have used that $\displaystyle \langle H, X\rangle =-\frac{1}{C}$ because $X:\Sigma\to \erre^{n+m}$ is a $C$-soliton of the IMCF.
Therefore,
$$
\frac{1}{C}>n-2-\epsilon
$$
for any $\epsilon>0$. Then $\displaystyle \frac{1}{C}\geq n-2$.

Now, let us consider the test function $v: \Sigma \rightarrow \erre$ defined as $v(p):=-\Vert X(p)\Vert^2=-r^2(p)$. If $v$ is constant in $\Sigma$, (\emph{i.e.}, $v(p)=-R^2$ for all $p \in \Sigma$), then $X:\Sigma^n\to \erre^{n+m}$ is a complete $C$-soliton for the IMCF such that $x(\Sigma)\subseteq S^{n+m-1}(R)$. Hence, applying Proposition \ref{teorSpherical}, $C=\frac{1}{n}$ and $\Sigma$ is minimal in the sphere $S^{n+m-1}(R)$.

On the other hand, if $v$ is non constant on $\Sigma$, as $ \sup\Sigma v < \infty$, $v \in C^\infty(\Sigma)$ and $\Sigma$ is parabolic, we apply Theorem \ref{teo0} to obtain a sequence $\{x_k\}\subset \Sigma$ such that, using Lemma \ref{lemma}:
$$
\Delta v_{\epsilon}(x_k)=-2(n-\frac{1}{C})<0 \,\,\,\forall k \in \ene
$$
\noindent and hence, $n > \frac{1}{C}$, and the Theorem is proved.

Let us suppose now that  $X:\Sigma^n\to\erre^{n+m}$ is a complete and non-compact, parabolic self-expander for the IMFC with  $C=\frac{1}{n}$. Then,  using Lemma \ref{lemma}:
$$
 \triangle^\Sigma v(x)=-2(n-\frac{1}{C}) =0
$$
\noindent As $ \sup_\Sigma v < \infty$, $v \in C^\infty(\Sigma)$ and $\Sigma$ is parabolic, then $v$, and hence $r$ are constant on $\Sigma$. Applying Proposition  \ref{teorSpherical},  $X:\Sigma^n\to S^{n+m-1}(R)$ is minimal  for some radius $R>0$. 

Namely, parabolic self-expanders with velocity $C =\frac{1}{n}$ always realizes as minimal submanifolds of a sphere of some radius.
\end{proof}

As Corollaries of Theorem \ref{teorNecIMCF}, we have that $2$-dimensional self-shrinkers for IMCF are non-parabolic and that, when $n \geq 3$, self-shrinkers and self-expanders with velocity $C >\frac{1}{n-2}$ are non-parabolic.

\begin{corollary}\label{cor5}
 Let $X:\Sigma^n \to \erre^{n+m}$ be a complete and non-compact soliton for the IMCF. Then
 \begin{enumerate}
 \item If $n =2 $ and $C < 0$, $\Sigma^n$ is non-parabolic.
 \item If $n \geq 3 $ and $C < 0$ or $C >\frac{1}{n-2}$, $\Sigma^n$ is non-parabolic.
 \end{enumerate}
\end{corollary}

\begin{corollary}\label{cor6}
There are no complete, non-compact and smooth $1$-dimensional solitons for the IMCF with velocity $C \in (-1,1)$
\end{corollary}
\begin{proof}
When $n=1$, then, applying Theorem \ref{teorNecIMCF}, if $\Sigma$ is parabolic, then   
 $ -1 \leq\frac{1}{C}\leq 1$, so  $C \in (-\infty, -1] \cap [1, \infty)$.
  Hence, if $C \in (-1,1)$, then $\Sigma^1$ should be non-parabolic. 
  But $\Sigma^1$, complete, non-compact, and regular is conformally isometric to $\erre$ with the standard metric, which is parabolic. 
  This means that $C \in (-1,1)$ it is not an allowed velocity constant for a smooth $1$- soliton for IMCF.\end{proof}
 
 Finally, we shall follow the  argument used by M. Rimoldi in \cite{Rim} on solitons for the MCF, based in the application of Theorem \ref{k} to obtain an extension of previous Corollary to solitons for the IMCF with dimension $n > 1$.
 
\begin{corollary}\label{teorSuffParabIMCF} 
There are no complete, connected and non-compact properly immersed solitons for the IMCF, $X:\Sigma^n \to \erre^{n+m}$, with velocity $C \in ]0,\frac{1}{n}]$.
 \end{corollary}
\begin{proof}
If $C \in ]0,\frac{1}{n}]$, then $\Sigma$ is parabolic because Theorem  \ref{k}: In fact, given $v(p):=r^2_{\vec{0}}(p)=\Vert X(p) \Vert^2$, as $C \in ]0,\frac{1}{n}]$, then
$$\Delta v= 2(n-\frac{1}{C})\leq 0$$
Hence, $v$ is superharmonic outside a compact and $v(p) \to \infty$ when $p \to \infty$ because  $\Sigma$ is properly immersed. Using Theorem \ref{k}, $\Sigma$ is parabolic. Now, we apply Theorem \ref{teorNecIMCF} to conclude that $C \in [\frac{1}{n}, \frac{1}{n-2}]$. Hence, $C=\frac{1}{n}$, so $X:\Sigma^n\to S^{n}(R)$ is a spherical and minimal  isometric immersion for some radius $R>0$.  Therefore, $\Sigma$ is compact, which is a contradiction.
\end{proof}

\begin{remark}
As a consequence of the proof of Corollary \ref{teorSuffParabIMCF}, if $X:\Sigma^n\to \erre^{n+m}$ is a complete and non-compact properly immersed non-parabolic soliton for the IMCF, then $C <0$ or $C >\frac{1}{n}$, namely, if they exists, {\em all} complete and non-compact non-parabolic solitons for the IMCF with velocity $C \in (0, \frac{1}{n}]$  {\em are not} properly immersed.
\end{remark}
 
%%%%%%%%%%%%%%%%%%%%%%%%%%%%%
\section{Examples}\label{exam}
%%%%%%%%%%%%%%%%%%%%%%%%%%%%%%%
Along this section we will analyze how to apply the geometric characterizations of parabolicity  that appears in this paper for the case of generalized cylinders (example \ref{Cylexemp}) and on the other hand, example \ref{IldefLerma}, we will deduce geometric properties of the family of examples given in \cite{CL2} where $\mathbb{R}^2$ is conformally immersed as a self-expander of the MCF in $\mathbb{R}^4$.
%%%%%%%%%%%
\begin{example}[Generalized cylinders]\label{Cylexemp}
%%%%%%%%%%%%
 Given $\rho>0$ and $k\in \mathbb{N}$, the following hypersurface of $\erre^{n+1}$
 $$
 C_k(\rho):=\left\{(x_1,\ldots,x_{n+1})\in \erre^{n+1}\, :\, x_1^2+\cdots+x_{k+1}^2=\rho^2\right\}
 $$
is called \emph{generalized cylinder}. The generalized cylinder $C_k(\rho)$ is isometric to $S^{k}(\rho)\times \mathbb{R}^{n-k}$, and the 
 inclusion map 
 $$X:C_k(\rho)\to\erre^{n+1},\quad  x=(x_1,\ldots,x_{n+1})\in C_k(\rho)\to X(x)=(x_1,\ldots,x_{n+1})$$ 
 is an immersion of $C_k(\rho)$ in $\mathbb{R}^{n+1}$. It is assumed that $0\leq k\leq n$. In the extreme cases, $k=0$ and $k=n$, the generalized cylinders $C_0(\rho)$ and $C_n(\rho)$ are isometric to  $\mathbb{R}^n$ and $S^n(\rho)$ respectively. Since $C_k(\rho)$ is isometric to $S^k(\rho)\times \erre^{n-k}$ and $S^k(\rho)$ is compact, it is known that  $C_k(\rho)$ is parabolic if and only if $n-k\leq 2$. In this example we will explain how deduce this behavior by using the geometric properties of $C_k(\rho)$. First of all we must remark that since the  mean curvature vector field of $X:C_k(\rho)\to \mathbb{R}^n$ is given by
 $$
 H=-\frac{k}{\rho^2}X^\perp.
 $$
 Then $X: C_k(\rho)\to\mathbb{R}^{n+1}$ can be considered as  $\lambda$-self-shrinker for the MCF with $\lambda=\frac{k}{\rho^2}$. Likewise, since
 $$
 \frac{H}{\Vert H\Vert^2}=-\frac{1}{k} X^\perp.
 $$
 the immersion $X:C_k(\rho)\to \mathbb{R}^{n+1}$ is a $C$-self-expander for the IMCF with $C=\frac{1}{k}$. Hence, $X:C_k(\rho)\to\mathbb{R}^{n+1}$ is at the same time a self-shrinker for the MCF and a self-expander for the IMCF. 
 
 %That allow us to apply the Theorems of this paper which deals with solitons of the MCF and the Theorems of this paper which deals with solitons of the IMCF. 
 
 Looking at necessary conditions for parabolicity of solitons of the IMCF, by applying Theorem \ref{teorNecIMCF} to $X:C_k(\rho)\to \mathbb{R}^{n+1}$, we can deduce that if $C_k(\rho)$ is parabolic then
 $$
 0\leq n-k\leq 2
 $$
 with $k=n$ if $C_k(\rho)$ is a minimal immersion of a sphere. This is obvious in this case because $C_n(\rho)$ is the $n$-sphere of radius $\rho$ in $\erre^{n+1}$.  Moreover, if $n-k >2$, then $C_k(\rho)$ is non-parabolic. 

Now, we can look for sufficient conditions for parabolicity of solitons of the MCF. 
The first geometric conclusion is  that since $X$ is a proper immersion, by Theorem \ref{teorSuffMCF} we know the following relations for the weighted volumes
\begin{enumerate}
\item $\displaystyle\int_{C_k(\rho)}e^{-\lambda\frac{r^2}{2}}dV<\infty$.
\item $\displaystyle\int_{C_k(\rho)}r^2e^{-\lambda\frac{r^2}{2}}dV<\infty$.
\item $\lambda\displaystyle\int_{C_k(\rho)}e^{-\lambda\frac{r^2}{2}}dV=n\displaystyle\int_{C_k(\rho)}r^2e^{-\lambda\frac{r^2}{2}}dV$.  
\end{enumerate}
It is tedious but not difficult to check the above statements for $X:C_k(\rho)\to\mathbb{R}^{n+1}$. In fact if $n-k\geq 2$,
 $$\begin{aligned}
 \int_{C_k(\rho)}e^{-\lambda\frac{r^2}{2}}dV=& \frac{e^{-k/2}\rho^{n-k}\Gamma\left[\frac{n-k}{2}\right]2^{\frac{n-k-2}{2}}}{k^{\frac{n-k}{2}}}{\rm vol}(S^{k}_1)\cdot {\rm vol}(S^{n-k-1}_1)   
   \end{aligned}
 $$
 and 
 $$
 \begin{aligned}
 \int_{C_k(\rho)}r^2e^{-\lambda\frac{r^2}{2}}dV=& \frac{e^{-k/2}n\rho^{n-k+2}\Gamma\left[\frac{n-k}{2}\right]2^{\frac{n-k-2}{2}}}{k^{\frac{n-k}{2}+1}}{\rm vol}(\mathbb{S}^{k}_1)\cdot {\rm vol}(\mathbb{S}^{n-k-1}_1)   
 \end{aligned}
 $$
 Then 
 $$
 \frac{\int_{C_k(\rho)}e^{-\lambda\frac{r^2}{2}}dV}{\int_{C_k(\rho)}r^2e^{-\lambda\frac{r^2}{2}}dV}=\frac{k}{n\rho^2}=\frac{\lambda}{n}
 $$
 as it is predicted by statement (3) of Theorem \ref{teorSuffMCF}. 
 
 On the other hand, we can to apply Theorem \ref{teorSuffMCF2}
 to see that $C_k(\rho)$ is parabolic if $n-k=2$. To do it, we have to prove that 
$$
\begin{aligned}
\int^\infty \frac{te^{-\frac{\lambda}{2}t^2}}{\Psi_{C_k(\rho)}(t)}dt
=&\infty
\end{aligned}
$$ 
where
$$
\begin{aligned}
\Psi_{C_k(\rho)}(R)=&{\rm vol}(\mathbb{S}^{k}_1)\cdot {\rm vol}(\mathbb{S}^{n-k-1}_1)\cdot \int_{\sqrt{R^2-\rho^2}}^\infty \left(t^2+\rho^2\right) t^{n-k-1}e^{-\frac{k}{2\rho^2}\left(t^2+\rho^2\right)}dt\\
=&{\rm vol}(S^{k}_1)\cdot {\rm vol}(S^{n-k-1}_1)\cdot \int_R^\infty z^3e^{-\frac{k}{2\rho^2}z^2}\left(z^2-\rho^2\right)^\frac{n-k-2}{2}dz
\end{aligned}
$$
In the case $n-k=2$, we have that
$$
\Psi_{C_k(\rho)}(R)=\frac{2\rho^4e^{-\frac{k}{2\rho^2}R^2}}{k^2}{\rm vol}(S^{k}_1)\cdot {\rm vol}(S^{n-k-1}_1)\left(\frac{k}{2\rho^2}R^2+1\right)
$$
Hence finally,
$$
\begin{aligned}
\int^\infty \frac{te^{-\frac{\lambda}{2}t^2}}{\Psi_{C_k(\rho)}(t)}dt=&\frac{1}{\frac{2\rho^4e^{-\frac{k}{2\rho^2}R^2}}{k^2}{\rm vol}(S^{k}_1)\cdot {\rm vol}(\mathbb{S}^{n-k-1}_1)}\int^\infty \frac{t}{\frac{k}{2\rho^2}t^2+1}dt\\
=&\frac{1}{\frac{k}{\rho^2}\frac{2\rho^4e^{-\frac{k}{2\rho^2}R^2}}{k^2}{\rm vol}(\mathbb{S}^{k}_1)\cdot {\rm vol}(\mathbb{S}^{n-k-1}_1)}\lim_{t\to\infty}\log\left(\frac{k}{2\rho^2}t^2+1\right)\\
=&\infty
\end{aligned}
$$
anb by using Theorem \ref{teorSuffMCF2} $C_k(\rho)$ is a parabolic manifold.
\end{example}
%%%%%%%%%%%%%%%%%%%%%%%
\begin{example}[Parabolic $2$-dimensional self-expanders]\label{IldefLerma}\
%%%%%%%%%%%%%%%%

In the following example we will show how to deduce geometric properties from the conformal type of a soliton, applying  our Corollary \ref{cor8}. I. Castro and A. Lerma have constructed in \cite{CL2} the following $2$-dimensional $\lambda$ self-expander immersed in $\mathbb{R}^4=\mathbb{C}^2$ for any $\delta>0$
$$
X_\delta: \mathbb{R}^2\to \mathbb{C}^2,\quad X_\delta(s,t):=\frac{1}{\sqrt{-2\lambda}}\left(i{\rm s}_\delta\cosh(t)e^{-\frac{i s}{{\rm c}_\delta}},\, {\rm t}_\delta\sinh(t)e^{i{\rm c}_\delta s}\right),
$$
with ${\rm s}_\delta=\sinh(\delta)$, ${\rm c}_\delta=\cosh(\delta)$ and ${\rm t}_\delta=\tanh(\delta)$. This is a conformal immersion of $\mathbb{R}^2$ into $\mathbb{C}^2$. 

As the parabolicity is preserved on conformal changes of the metric, and the immersion $X_\delta$ is conformal for any $\delta>0$, $(\mathbb{R}^2,X_\delta^*\left(g^{\rm can}_{\mathbb{C}^2}\right))$ is parabolic for any $\delta>0$, where $X_\delta^*\left(g^{\rm can}_{\mathbb{C}^2}\right)$ is the pull-back of the canonical metric of $\mathbb{C}^2$ given by $X_\delta$.  The conformal type of $(\mathbb{R}^2,X_\delta^*\left(g^{\rm can}_{\mathbb{C}^2}\right))$ implies certain behavior of the mean curvature vector field. More precisely,  according to our Corollary \ref{cor8}, since   $(\mathbb{R}^2,X_\delta^*\left(g^{\rm can}_{\mathbb{C}^2}\right))$ is parabolic, the infimum of the norm of the mean curvature vector field of $X_\delta$ is therefore $0$ with independence on $\delta$. In fact, the mean curvature vector field can be explicitly computed as, (see proof of Proposition 2 of \cite{CL2}), 
$$
\vec{H}(s,t)=\frac{{\rm s}_\delta^2e^{-2u(t)}}{2{\rm c}_\delta}J\left(\frac{\partial}{\partial s}X_{\delta}(s,t)\right)
$$
where $J$ is the complex structure on $\mathbb{C}^2$ and $u(t)=\ln \left(\frac{1}{-2\lambda}\left({\rm t}_\delta^2\cosh^2(t)+{\rm s}_\delta^2\sinh^2(t)\right)\right)$. Then
\begin{equation*}
\begin{aligned}
\vec{H}(s,t)=&\frac{{\rm s}_\delta^2}{2{\rm c}_\delta}\frac{4\lambda^2}{\left({\rm t}_\delta^2\cosh^2(t)+{\rm s}_\delta^2\sinh^2(t)\right)^2}\left(i{\rm t}_\delta\cosh(t)e^{-\frac{is}{{\rm c}_\delta}},\, -{\rm s}_\delta \sinh(t)e^{i{\rm c}_\delta s}\right)
\end{aligned}
\end{equation*}
and it is easy to check that,
$$
\lim_{t\to \infty} \vec{H}(s,t)=\vec{0}.
$$
Hence 
$$
\inf\Vert \vec{H}\Vert=0.
$$

\end{example}
%
%%%%%%%%%%%%%%%%%%%%%%%%%%%%%%%%%%%%%%%%%%%%%%%%%%%
%%%%%%%%%%%%%%%%%%%%
\section{Solitons confined in a ball}\label{confined}
%%%%%%%%%%%%%%%%%%%%%%
%%%%%%%%%%%%%%%
\subsection{Solitons for MCF confined in a ball}\label{confinedMCF}
%%%%%%%%%%%%
We are going to see,  in the spirit  of the results in \cite{PiRi}, (see Proposition 5), that parabolic self-shrinkers for the MCF, $X:\Sigma^n\to\erre^{n+m}$,  confined in a ball of radius $\sqrt{\frac{n}{\lambda}}$ realizes as minimal submanifolds of the sphere $\mathbb{S}^{n+m-1}(\sqrt{\frac{n}{\lambda}})$. 
\begin{theorem}\label{teorConfMCF}
Let $X:\Sigma^n\to\erre^{n+m}$ be a complete $\lambda$-self-shrinker for the MCF with respect $\vec{0} \in \erre^{n+m}$, ( $\lambda > 0$). 
Let us suppose that $\Sigma$  is parabolic. Then:
\begin{enumerate}
\item Either exists a point $p \in \Sigma$ such that $r(p) > \sqrt{\frac{n}{\lambda}}$
\item or  $X(\Sigma) \subseteq \mathbb{S}^{n+m-1}(\sqrt{\frac{n}{\lambda}})$, and $X:\Sigma^n\to\mathbb{S}^{n+m-1}(\sqrt{\frac{n}{\lambda}})$ is minimal. 
\end{enumerate}
\end{theorem}
\begin{remark}
Namely, there are no complete  parabolic $\lambda$-self-shrinkers for the MFC inside the interior of a ball of radius $R \leq \sqrt{\frac{n}{\lambda}}$. If they are confined \emph{i.e.}, $X(\Sigma) \subseteq B^{n+m}_R(\vec{0})$, with $R \leq \sqrt{\frac{n}{\lambda}}$ , then $\Sigma$ realizes as minimal submanifolds of the sphere $\mathbb{S}^{n+m-1}(\sqrt{\frac{n}{\lambda}})$.
\end{remark}
  \begin{proof}
  If there is no point $p \in \Sigma$ such that $r(p) > \sqrt{\frac{n}{\lambda}}$, then $r(p) \leq  \sqrt{\frac{n}{\lambda}} \,\,\forall p \in \Sigma$, so we have that $X(\Sigma) \subseteq B^{n+m}_R(\vec{0})$  with $R \leq \sqrt{\frac{n}{\lambda}}$.
  
 Let us consider the function $u: \Sigma \rightarrow \erre$ defined as $u(p):=\Vert X(p)\Vert^2=r^2(p)$. We assume by hypothesis, ($X(\Sigma) \subseteq B^{n+m}_R(\vec{0})$), that 
$$ \sup_\Sigma u < \infty.$$
\noindent Moreover, since $X(\Sigma) \subseteq B^{n+m}_R(\vec{0})$, with $R \leq \sqrt{\frac{n}{\lambda}}$, we have that $ \Vert X^\bot \Vert^2 \leq  \Vert X \Vert^2 \leq  \frac{n}{\lambda}$. Then, using Lemma \ref{lemma}:
$$
 \triangle^\Sigma u(x)=2(n-\lambda \Vert X^\bot \Vert^2) \geq 0
 $$

\noindent Then, as $\Sigma$ is parabolic, we conclude that $u$ is constant on $\Sigma$, so $r^2(x)=R^2\,\,\forall x \in \Sigma$, for some $R \leq \sqrt{\frac{n}{\lambda}}$, (because $X(\Sigma) \subseteq B^{n+m}_{\sqrt{\frac{n}{\lambda}}}(\vec{0})$).  Hence $X(\Sigma) \subseteq S^{n+m-1}(R)$. 

On the other hand, as $X(x) \in T_xS^{n+m-1}(R)^\bot \subseteq T_x\Sigma^\bot\,\,\forall x \in \Sigma$, then $X=X^\bot$ and $X^T=0$. But, as $u$ is constant on $\Sigma$ and $X=X^\bot$, then 
$$
 \triangle^\Sigma u(x)=2(n-\lambda \Vert X \Vert^2) = 0
 $$
 \noindent and therefore, $R^2=r^2(x)=\Vert X \Vert^2 =\frac{n}{\lambda}$. Hence $X(\Sigma) \subseteq S^{n+m-1}(\sqrt{\frac{n}{\lambda}})$ and, by Proposition \ref{teorSpherical}, $\Sigma$ is minimal in $S^{n+m-1}(\sqrt{\frac{n}{\lambda}})$.\end{proof}

 \begin{corollary}\label{cor10}
Let $X:\Sigma^n\to\erre^{n+1}$ be a complete and connected  self-shrinker for the MCF, with $\lambda > 0$. Let us suppose that $\Sigma^n$ is parabolic and $X(\Sigma) \subseteq B^{n+1}_R(\vec{0})$, with $R \leq \sqrt{\frac{n}{\lambda}}$.

Then  $$\Sigma^n\equiv S^{n}\left(\sqrt{\frac{n}{\lambda}}\right)$$ 
 \end{corollary}
 \begin{proof}
 In Theorem above, we have proved that $\Vert X \Vert^2 =\frac{n}{\lambda}$ on $\Sigma$.
 Hence $X: \Sigma^n \to S^{n}(\sqrt{\frac{n}{\lambda}})$ is a local isometry and therefore, as $\Sigma$ is connected and complete and $S^{n}(\sqrt{\frac{n}{\lambda}})$ is connected, then $X$ is a Riemannian covering, (see \cite{S}, p. 116). Moreover, as $S^{n}(\sqrt{\frac{n}{\lambda}})$ is simply connected, then $X$ is an isometry, (see \cite{Lee}, Corollary 11.24).
 \end{proof}
 \begin{remark}
 If $n>2$, it is enough to assume that $X:\Sigma^n\to\erre^{n+1}$ is a complete and connected  soliton for the MCF, by virtue of Theorem \ref{teorNecMCF}.
 \end{remark}
 
 Finally, we shall see that it is not possible to find complete and non-compact parabolic self-expanders confined in a ball
  \begin{theorem}\label{noselfexpanders}
 There are not complete and non-compact parabolic self-expanders for MCF $X: \Sigma^n \to \erre^{n+m}$ confined in a ball.
 \end{theorem}
 \begin{proof}
Let us consider $X: \Sigma^n \to \erre^{n+m}$ a complete and non-compact parabolic self-expander. As $\Sigma$ is parabolic, then $n=1$ or $n=2$ by Corollary \ref{cor7}. On the other hand, as $\lambda <0$, we have, on $\Sigma$:
$$\Delta^\Sigma r^2=2n-2\lambda \Vert X^\bot\Vert^2 \geq 0$$
Let us suppose that $X(\Sigma) \subseteq B^{n+m}_R(\vec{0})$ for some $R >0$. Then, as $\sup_{\Sigma} r^2 \leq R <\infty$ and $\Sigma$ is parabolic, $r$ is constant on $\Sigma$, so $X(\Sigma) \subseteq S^{n}(R_0)$ with $R_0 \leq R$. By Proposition \ref{teorSpherical}, $\lambda =\frac{n}{R_0^2} >0$, ($n=1$ or $n=2$), which is a contradiction.\end{proof}

%%%%%%%%%%%%%%%
\subsection{Solitons for IMCF confined in a ball}\label{confinedIMCF}\
%%%%%%%%%%%%

Our aim in this subsection is the same than in subsection \S 6.1: we are going to see that parabolic self-expanders for the IMCF included in a $R$-ball, $X(\Sigma) \subseteq B^{n+m}_R(\vec{0})$, realizes as minimal submanifolds of a $r_0$-sphere with $r_0 \leq R$, and its velocity, (which do not depends on the radii $r$ and $R$), must be $C=\frac{1}{n}$ in this case.

\begin{theorem}\label{teorConfIMCF2}
Let $X:\Sigma^n\to\erre^{n+m}$ be a complete and non-compact soliton for the IMFC. Let us suppose that $\Sigma$ is parabolic.  Then:
\begin{enumerate}
\item Either $\Sigma$ is unbounded
\item or $C=\frac{1}{n}$, $X(\Sigma) \subseteq S^{n+m-1}(r_0)$ with $r_0 \leq R$ and $X:\Sigma^n\to S^{n+m-1}(r_0)$ is minimal. 
\end{enumerate}
\end{theorem}
\begin{proof}
Let us suppose that $\Sigma$ is bounded, \emph{i.e.}, it is confined in a ball  $X(\Sigma) \subseteq B^{n+m}_R(\vec{0})$.
 We are going to apply Theorem \ref{teo0} to the  function,
 $$
 u(x):=r^2(x)=\Vert X(x)\Vert^2
 $$
 Suppose that $u$ is nonconstant on $\Sigma$. Since $\sup_\Sigma u<\infty$, as $\Sigma$ is parabolic,  by Theorem \ref{teo0} there exists a sequence $\{x_k\}$ such that
 $$
 \triangle u(x_k)<0
 $$
 but by Lemma \ref{lemma} ,
\begin{equation}\label{equation}
 0>\triangle u(x_k)=2n-\frac{2}{C}
 \end{equation}
 Hence, for any $x\in \Sigma^n$
 $$
 \triangle u(x)=2n-\frac{2}{C}<0
 $$

On the other hand, let us consider the function,
 $$
 v(x):=-r^2(x)=-\Vert X(x)\Vert^2
 $$
\noindent If we assume that $v$ is nonconstant on $\Sigma$ and since $\sup_\Sigma u\leq 0 <\infty$, we have that, applying again Theorem \ref{teo0}
there exists a sequence $\{x_k\}$ such that
 $$
 \triangle v(x_k)=-2n+\frac{2}{C}<0
 $$
 \noindent Hence, $\frac{1}{n}<C<\frac{1}{n}$, so $u$ and $v$ must be constant functions. Therefore, applying Proposition \ref{teorSpherical}, $C=\frac{1}{n}$,  and $X:\Sigma^n\to S^{n+m-1}(r_0)$ is minimal.\end{proof}

As a corollary, and taking into account that every compact manifold is parabolic, we have the following result due to I. Castro and A. Lerma in \cite{CL1} 
\begin{corollary}\label{cor9}
 Let $X:\Sigma^n\to\erre^{n+m}$ be a complete soliton for the IMFC. Suppose that $\Sigma^n$ is compact. Then, $C=\frac{1}{n}$, and $X(\Sigma^n)$ is contained in 
 a sphere $S^{n+m-1}(R)\subset \erre^{n+m}$ of some radius $R$ centered at the origin of $\erre^{n+m}$. 
 Moreover, $X:\Sigma^n:\to S^{n+m-1}(R)\subset \erre^{n+m}$ is a minimal immersion into $S^{n+m-1}(R)$.
\end{corollary}
Another Corollary is the  following analogous to Corollary \ref{cor10} for parabolic and confined self-shrinkers for the MCF:
\begin{corollary}\label{cor11}
 Let $X:\Sigma^n\to\erre^{n+1}$ be a connected and complete soliton for the IMFC. Let us suppose that $\Sigma^n$ is parabolic and $X(\Sigma) \subseteq B^{n+1}_R(\vec{0})$, for some $R>0$.
Then  $$\Sigma^n\equiv S^{n}(R)$$ 
 \end{corollary}
 \begin{proof}
 As $X(\Sigma) \subseteq B^{n+1}_R(\vec{0})$, for some $R>0$, we have, applying Theorem \ref{teorConfIMCF2}, that 
 $C=\frac{1}{n}$, $X(\Sigma) \subseteq S^{n}(r_0)$ with $r_0 \leq R$ and $X:\Sigma^n\to\mathbb{S}^{n}(r_0)$ is minimal. 
 
 Hence $X: \Sigma^n \to \mathbb{S}^{n}(\sqrt{\frac{n}{\lambda}})$ is a local isometry and therefore, as $\Sigma$ is connected and complete and $\mathbb{S}^{n}(\sqrt{\frac{n}{\lambda}})$ is connected, then $X$ is a Riemannian covering, (see \cite{S}, p. 116). Moreover, as $\mathbb{S}^{n}(\sqrt{\frac{n}{\lambda}})$ is simply connected, then $X$ is an isometry, (see \cite{Lee}, Corollary 11.24).
 \end{proof}
%%%%%%%%%%%%%%%%%%%%%%%%%%%%%%%%%%%%%%%%%%%%%%%%%%
%%%%%%%%%%%%%%%%%%%%%%%%%%%%%%%%%%%%%%%%%%%%%%%%
%%%%%%%%%%%%%%%
\section{Mean Exit Time, and Volume of MCF-solitons}\label{meanexit}\
%%%%%%%%%%%%

The Mean Exit Time function for the Brownian motion defined on a precompact domain of the manifold satisfies a Poisson 2nd order PDE equation with Dirichlet boundary data, which, trhough the application of the divergence theorem, provides some infomation about the volume growth of the manifold. In the next sections and subsections we will explore these questions for MCF and IMCF solitons.
%%%%%%%%%%%%%%%%%%%%%%%%%%%%%%%%%%%%%%%%%%%%%%%%%%%
\subsection{Mean Exit time on Solitons for MCF}\label{moments}\
%%%%%%%%%%%%%%%%%%%%%%%%%%%%%%%%%%%%%%%%%%%%%%%%%%

Let $X: \Sigma^n \rightarrow \erre^{n+m}$ be an $n$-dimensional  $\lambda$-soliton in $\erre^{n+m}$ for the Mean Curvature Flow, (MCF), with respect $\vec{0} \in \erre^{n+m}$. Let us consider $r:\Sigma\to \mathbb{R}$  the {\em{extrinsic distance function}} from $\vec{0}$ in $\Sigma^n$. 
Given the extrinsic ball $D_R(\vec{0})=X^{-1}(B_R^{n+m}(\vec{0}))$, let us consider the Poisson problem 
\begin{equation}  \label{poisson}
\left\{
\begin{array}{lcc}
\Delta^\Sigma E+1= 0&{\rm on}& D_R, \\
E=0& {\rm on}& \partial D_R.
\end{array}\right.
\end{equation}
The solution of the Poisson problem on a geodesic $R$-ball $B^n(R)$ in $\erre^n$
\begin{equation}\label{poissoneuclidean}
\left\{\begin{array}{lcc}
\Delta E+1=0&{\rm on}& B_R^n(R)\\
E=0&{\rm on}&S^{n-1}(R)
\end{array}\right.
\end{equation}
is given by the radial function $ E^{0,n}_R(r)= \frac{R^2-r^2}{2n}$.

Let us denote $E_R$ the solution of (\ref{poisson}) in $D_R \subseteq \Sigma$. Transplanting the radial solution  $ E^{0,n}_R(r)$ to the extrinsic ball by mean the extrinsic distance function, we have $\bar {E}_R: D_R \rightarrow \erre$ defined as $\bar {E}_R(p):= E^{0,n}_R(r(p))$.

Our first result is a comparison for the Mean Exit Time function:

\begin{proposition}\label{meanexitShri}
Let $X: \Sigma^n \rightarrow \erre^{n+m}$ be a properly immersed $\lambda$-soliton for the MCF, with respect $\vec{0} \in \erre^{n+m}$. Let us suppose that $X(\Sigma) \not\subseteq S^{n+m-1}(R)$ for any radius $R>0$. Given the extrinsic ball $D_R(\vec{0})$, we have

\begin{enumerate}
\item If $\lambda \geq 0$,

 \begin{equation*}
\bar {E}_R(x) \leq E_R(x),\quad \forall x \in D_R. 
\end{equation*}
\item Or if $\lambda \leq 0$,
\begin{equation*}
\bar {E}_R(x)  \,\geq E_R(x),\quad\forall x \in D_R
\end{equation*}
\end{enumerate}
\end{proposition}

\begin{proof}
We have, as $\bar{E}_R(x):=E^{0,n}_R(r(x))= \frac{R^2-r(x)^2}{2n}$ and applying Lemma \ref{lemma}, that, on $D_R$
\begin{equation}\label{eq.5}
\begin{aligned}
\Delta^\Sigma \bar{E}_R=&\left(\bar{E}''_R(r)-\bar{E}'_R(r)\frac{1}{r}\right)\Vert \nabla^\Sigma r\Vert^2\\&+\Bar{E}'_r(r)\left(\frac{n}{r}+\langle \nabla^{\erre^{n+m}} r, \vec{H}_\Sigma\rangle\right)
=-1-\frac{1}{n}\langle r \nabla^{\erre^{n+m}} r, \vec{H}_\Sigma\rangle
\end{aligned}
\end{equation}
On the other hand, $X(p)=r(p)\nabla^{\erre^{n+m}}r(p)$ for all $p \in \Sigma$, and, moreover, as we have that $\vec{H}_\Sigma(p)=-\lambda X^{\bot}(p)\,\,\,\forall p \in \Sigma$, then 
$$\langle r \nabla^{\erre^{n+m}} r, \vec{H}_\Sigma\rangle=-\lambda\Vert X^{\bot}\Vert=-\frac{\Vert\Vec{H}_\Sigma\Vert^2}{\lambda}$$
\noindent
Therefore, if $\lambda \geq 0$, we obtain
\begin{equation}\label{eq.6}
%\begin{aligned}
\Delta^\Sigma \bar{E}_R=
-1+\frac{1}{n}\frac{\Vert\Vec{H}_\Sigma\Vert^2}{\lambda} \geq -1=\Delta^\Sigma E_R
%\end{aligned}
\end{equation}

As $\bar{E}_R=E_R$ on $\partial D_R$, we apply now the Maximum Principle to obtain the inequality
$$\bar{E}_R \leq E_R$$
\noindent Inequality (2) follows in the same way.\end{proof}

%%%%%%%%%%%%%%%%%%%%%%%%%%%%
\subsection{Volume of Self-shrinkers for MCF}\label{volumeSS}\
%%%%%%%%%%%%%%%%%%%%%%%%%%%%%%%%

As a consequence or the Proposition \ref{meanexitShri}, and using the Divergence theorem we have the following isoperimetric inequality. 

\begin{theorem}\label{isopShri}
Let $X: \Sigma^n \rightarrow \erre^{n+m}$ be a complete properly immersed  $\lambda$-self-shrinker in $\erre^{n+m}$ for the MCF, with respect $\vec{0} \in \erre^{n+m}$.  Let us suppose that $X(\Sigma) \not\subseteq S^{n+m-1}(R)$ for any radius $R>0$. Given the extrinsic ball $D_R(\vec{0})=\Sigma \cap B_R^{n+m}(\vec{0})$, we have
\begin{equation} \label{isopComp2}
\frac{\Vol(\partial D_R)}{\Vol(D_R)} \geq
\Big(1-\frac{1}{n\lambda}\frac{\int_{D_R} H_\Sigma^2}{\Vol(D_R)}\Big)\frac{\Vol(S^{n-1}_R)}{\Vol(B^{n}_R)} \,\,\,\,\,\textrm{for all}\,\,\,
R>0  \quad 
\end{equation}
where 
\begin{equation}\label{factor}
1-\frac{1}{n\lambda}\frac{\int_{D_R} H_\Sigma^2}{\Vol(D_R)} \geq 0 \,\,\forall R>0
\end{equation}
\end{theorem}

\begin{proof}

We are going to prove first that

\begin{equation}\label{eq.7}
1-\frac{1}{n\lambda}\frac{\int_{D_R} H_\Sigma^2}{\Vol(D_R)} \geq 0 \,\,\forall \,\,\, R >0
\end{equation}

\noindent To do that, let us consider the function $r^2: \Sigma \rightarrow \erre$, defined as $r^2(p)=\Vert X(p)\Vert^2$, where $r$ is the extrinsic distance to $\vec{0}$ in $\Sigma \subseteq \erre^{n+m}$. Then, applying Lemma \ref{lemma} to the radial function $F(r)=r^2$
\begin{equation}\label{eq.8}
\Delta^\Sigma r^2=2n+2\langle r \nabla^{\erre^{n+m}} r, \vec{H}_\Sigma\rangle.
\end{equation}
\noindent Taking into account that  $\langle r \nabla^{\erre^{n+m}} r, \vec{H}_\Sigma\rangle=-\lambda\Vert X^{\bot}\Vert=-\frac{\Vert\Vec{H}_\Sigma\Vert^2}{\lambda}$
we obtain
\begin{equation}\label{eq.9}
\Delta^\Sigma r^2=2n-2\frac{\Vert\Vec{H}_\Sigma\Vert^2}{\lambda}
\end{equation}
\noindent and hence
\begin{equation}\label{eq.10}
\Vert\Vec{H}_\Sigma\Vert^2=n\lambda-\frac{\lambda}{2}\Delta^\Sigma r^2.
\end{equation}
Integrating on $D_R$ equality above, and arranging terms, we have 
\begin{equation}\label{eq.11}
n\lambda  \Vol(D_R)-\int_{D_R}\Vert\Vec{H}_\Sigma\Vert^2d\sigma=\frac{\lambda}{2}\int_{D_R}\Delta^\Sigma r^2d\sigma
\end{equation}
Now we apply Divergence theorem taking into account that the unitary normal to $\partial D_R$ in $\Sigma$, pointed outward is $\mu=\frac{\nabla^\Sigma r}{ \Vert \nabla^\Sigma r\Vert}$ and the fact that $\nabla^\Sigma r=\frac{X^T}{\Vert X^T\Vert}$, 
\begin{equation}\label{eq.12}
\begin{aligned}
\int_{D_R}\Delta^\Sigma r^2d\sigma&=\int_{\partial D_R}\langle \nabla^\Sigma r^2, \frac{\nabla^\Sigma r}{ \Vert \nabla^\Sigma r\Vert}\rangle d\mu\\&=\int_{\partial D_R} 2 r\Vert \nabla^\Sigma r\Vert d\mu=2\int_{\partial D_R} \Vert X^T\Vert d\mu
\end{aligned}
\end{equation}
\noindent so equation (\ref{eq.11}) becomes
\begin{equation}\label{eq.13}
n\lambda  \Vol(D_R)-\int_{D_R}\Vert\Vec{H}_\Sigma\Vert^2d\sigma=\lambda\int_{\partial D_R} \Vert X^T\Vert d\mu
\end{equation}
\noindent and hence
\begin{equation}\label{eq.14}
0 \leq \frac{\int_{D_R}\Vert\Vec{H}_\Sigma\Vert^2d\sigma}{\Vol(D_R)}=n\lambda-\frac{\lambda\int_{\partial D_R} \Vert X^T\Vert d\mu}{\Vol(D_R)}\leq n\lambda .
\end{equation}
\noindent which implies inequality (\ref{eq.7}).
On the other hand, integrating on $D_R$ the first equality in (\ref{eq.6})
%$$-\Delta^\Sigma \bar{E}_R=1-\frac{1}{n}\frac{\Vert\Vec{H}_\Sigma\Vert^2}{\lambda}$$
we obtain
\begin{equation}\label{eq.15}
\begin{aligned}
-\int_{D_R}\Delta^\Sigma \bar{E}_Rd\sigma&=\int_{D_R}\Big(1-\frac{1}{n}\frac{\Vert\Vec{H}_\Sigma\Vert^2}{\lambda}\Big) d\sigma\\&=\Vol(D_R)-\frac{1}{n\lambda}\int_{D_R}\Vert\Vec{H}_\Sigma\Vert^2d\sigma.
\end{aligned}
\end{equation}

Now, applying Divergence Theorem, and taking into account, as before,  that the unitary normal to $\partial D_R$ in $\Sigma$, pointed outward is $\mu=\frac{\nabla^\Sigma r}{ \Vert \nabla^\Sigma r\Vert}$, we have
\begin{equation}\label{eq.16}
 -\int_{D_R}\Delta^\Sigma \bar{E}_Rd\sigma=-\bar{E}_R'(R)\int_{\partial D_R} \Vert \nabla^\Sigma r\Vert d\sigma \leq \frac{\Vol(B^{0,n}_R)}{\Vol(S^{0,n-1}_R)} \Vol(\partial D_R)
\end{equation}
\noindent Hence 
\begin{equation}\label{eq.17}
\Vol(D_R)-\frac{1}{n\lambda}\int_{D_R}\Vert\Vec{H}_\Sigma\Vert^2d\sigma \leq \frac{\Vol(B^{0,n}_R)}{\Vol(S^{0,n-1}_R)} \Vol(\partial D_R)
\end{equation}
\noindent so
\begin{equation}\label{eq.18}
\frac{\Vol(D_R)}{\Vol(\partial D_R)}  \leq \frac{\Vol(B^{0,n}_R)}{\Vol(S^{0,n-1}_R)} + \frac{1}{n\lambda}\frac{\int_{D_R}\Vert\Vec{H}_\Sigma\Vert^2d\sigma}{\Vol(\partial D_R)}
\end{equation}
and therefore for all $R>0$,
\begin{equation}\label{eq.19}
\frac{\Vol(\partial D_R)}{\Vol(D_R)} \geq
\Big(1-\frac{1}{n\lambda}\frac{\int_{D_R} \Vert \vec{H}_\Sigma\Vert^2 d\sigma}{\Vol(D_R)}\Big)\frac{\Vol(S^{n-1}_R)}{\Vol(B^{n}_R)}\end{equation}\end{proof}

%%%%%%%%%%%%%%%%%%%%%%%%%%%%
\subsection{Proper Self-shrinkers for MCF and their distance to the origin}\label{clasiSS}\
%%%%%%%%%%%%%%%%%%%%%%%%%%%%%%%%

By using the above Proposition \ref{prop3.4} and inequality (\ref{factor}) we can state the following theorem which give us a dual description of the behavior of the self-shrinker when we change the hypothesis of parabolicity for the assumption that it is properly immersed.

\begin{theorem}\label{teoCompactSphere}
 Let $X: \Sigma^n \rightarrow \erre^{n+m}$ be a complete properly immersed $\lambda$-self-shrinker in $\erre^{n+m}$ 
for the Mean Curvature Flow, (MCF), with respect $\vec{0} \in \erre^{n+m}$. Then:

\begin{enumerate}
\item Either there exists a point $p \in \Sigma$ such that $r(p) < \sqrt{\frac{n}{\lambda}}$
\item or  $\Sigma^n$ is compact and $X: \Sigma \rightarrow S^{n+m-1}(\sqrt{\frac{n}{\lambda}})$ is a minimal immersion.
\end{enumerate}
\end{theorem}
\begin{proof}
Let us suppose that $X(\Sigma) \not\subseteq S^{n+m-1}(R)$ for any radius $R>0$. If, in addition, there is no point $p \in \Sigma$ such that $r(p) < \sqrt{\frac{n}{\lambda}}$, then $r(p) \geq  \sqrt{\frac{n}{\lambda}}\,\,\, \forall p \in \Sigma$.
Now, let us suppose that $\inf_{\Sigma}r> \sqrt{\frac{n}{\lambda}}$. Then, for any $p\in \Sigma$, we have that $1-\frac{\lambda}{n}r^2(p) < 0$. Hence 
\begin{equation}
\int_{D_R}\left(1-\frac{\lambda}{n}r^2 \right)e^{\frac{\lambda}{2}\left(R^2-r^2\right) }d\sigma <0
\end{equation}

Now, we need the following

\begin{lemma}\label{prop3.4}

Let $X: \Sigma^n \rightarrow \erre^{n+m}$ be a complete properly immersed  $\lambda$-self-shrinker in $\erre^{n+m}$ 
for the MCF, with respect $\vec{0} \in \erre^{n+m}$.  Let us suppose that $X(\Sigma) \not\subseteq S^{n+m-1}(R)$ for any radius $R>0$. Given the extrinsic ball $D_R$, if 
 $\Vol(D_R)>0$, we have
\begin{equation}\label{eq.19.1}
1-\frac{\int_{D_R}\Vert H_\Sigma\Vert^2d\sigma}{n\lambda \Vol(D_R)}=\frac{\int_{D_R}\left(1-\frac{\lambda}{n}r^2 \right)e^{\frac{\lambda}{2}\left(R^2-r^2\right) }d\sigma}{\Vol(D_R)} 
\end{equation}
\end{lemma}
\begin{proof}By applying the divergence theorem,
 $$
 \int_{D_R}{\rm div}\left(e^{-\frac{\lambda}{2}r^2}\nabla r^2\right)d\sigma=\int_{\partial D_R}e^{-\frac{\lambda}{2}r^2}\langle \nabla r^2,\frac{\nabla r}{\Vert \nabla r\Vert}\rangle d\mu=2R^2 e^{-\frac{\lambda}{2}R^2}\int_{\partial D_R} \Vert \nabla r\Vert d\mu
 $$
 By equation (\ref{eq.13}) we know that
 $$
\frac{R}{n\Vol(D_R)} \int_{\partial D_R}\Vert \nabla r\Vert d\mu=1-\frac{\int_{D_R}\Vert H_\Sigma\Vert^2d\sigma}{\lambda n \Vol(D_R)}
 $$
 Hence,
 $$
 1-\frac{\int_{D_R}\Vert H_\Sigma\Vert^2d\sigma}{\lambda n \Vol(D_R)}=\frac{e^{\frac{\lambda}{2}R^2}}{2n \Vol(D_R)}\int_{D_R}{\rm div}\left(e^{-\frac{\lambda}{2}r^2}\nabla r^2\right)d\sigma
 $$
 Finally, the proposition follows taking into account that, see equation (\ref{eq3.6.2}),
 $$\begin{aligned} 
 {\rm div}\left(e^{-\frac{\lambda}{2}r^2}\nabla r^2\right)=&2e^{-\frac{\lambda}{2}r^2}\left(n-\lambda r^2\right)
 \end{aligned}
 $$\end{proof}
Now, applying inequality (\ref{factor}) in Theorem \ref{isopShri} and Lemma \ref{prop3.4}, we have
\begin{equation}
0\leq 1-\frac{\int_{D_R}\Vert H_\Sigma\Vert^2d\sigma}{n\lambda \Vol(D_R)}=\frac{\int_{D_R}\left(1-\frac{\lambda}{n}r^2 \right)e^{\frac{\lambda}{2}\left(R^2-r^2\right) }d\sigma}{\Vol(D_R)} < 0
\end{equation}
\noindent which is a contradiction. 

Hence, or  $X(\Sigma) \subseteq S^{n+m-1}(R)$ for some radius $R_0>0$, or $\inf_{\Sigma}r=\sqrt{\frac{n}{\lambda}}$. 

In the first case, we have that $X:\Sigma \to S^{n+m-1}(R_0)$ will be a spherical immersion and, by Proposition \ref{teorSpherical}, as $\Sigma$ is a $\lambda$-soliton for the MCF, then $X$ is minimal and $\lambda=\frac{n}{R{^2}_0}$, namely, $X: \Sigma \rightarrow S^{n+m-1}(\sqrt{\frac{n}{\lambda}})$ is a minimal immersion.

In the second case we shall conclude the same: if $\inf_{\Sigma}r=\sqrt{\frac{n}{\lambda}}$, then $\sqrt{\frac{n}{\lambda}} \leq r(p) $ for all $ p \in \Sigma$ and hence $1-\frac{\lambda}{n}r^2(p)  \leq 0 \,\, \forall p \in \Sigma$. Then by inequality (\ref{factor}) and equality (\ref{eq.19.1}) we have
\begin{equation}
0\leq 1-\frac{\int_{D_R}\Vert H_\Sigma\Vert^2d\sigma}{n\lambda \Vol(D_R)}=\frac{\int_{D_R}\left(1-\frac{\lambda}{n}r^2 \right)e^{\frac{\lambda}{2}\left(R^2-r^2\right) }d\sigma}{\Vol(D_R)} \leq 0
\end{equation}
\noindent Therefore, $1-\frac{\lambda}{n}r^2(p)  =0 \,\, \forall p \in \Sigma$, so $X(\Sigma) \subseteq S^{n+m-1}(\sqrt{\frac{n}{\lambda}})$, and hence $X: \Sigma \rightarrow S^{n+m-1}(\sqrt{\frac{n}{\lambda}})$ is a complete spherical immersion, and as the radius $R=\sqrt{\frac{n}{\lambda}}$, then by Proposition \ref{teorSpherical}, $\Sigma$ is minimal in the sphere $S^{n+m-1}(\sqrt{\frac{n}{\lambda}})$.

Finally, as $X: \Sigma^n \rightarrow \erre^{n+m}$ is proper, then $\Sigma = X^{-1}( S^{n+m-1}(\sqrt{\frac{n}{\lambda}}))$ is compact.\end{proof}

Another result which describes the position of properly immersed $\lambda$-self-shrinkers $\Sigma^n$ with respect the critical ball $B^{n+m}_{\sqrt{\frac{n}{\lambda}}}(\vec{0})$ is following. We must remark that the proof is based partially in the proof of Theorem \ref{teoCompactSphere} and that the same result has been proved in \cite{HPR} as a corollary of the fact that properly immersed $\lambda$-self-shrinkers of MCF are $h$-parabolic submanifolds of the Euclidean space $\erre^{n+m}$  weighted with the Gaussian density $e^{h(r)}$, $h(r)=-\frac{\lambda}{2} r^2$.

\begin{theorem}\label{ballcomp}
Let $X: \Sigma^n \rightarrow \erre^{n+m}$ be a complete properly immersed $\lambda$-self-shrinker in $\erre^{n+m}$ 
for the Mean Curvature Flow, (MCF), with respect $\vec{0} \in \erre^{n+m}$. Let us suppose that:

\begin{enumerate}
\item Either $\Sigma$ is confined into the ball $X(\Sigma) \subseteq B^{n+m}_{\sqrt{\frac{n}{\lambda}}}(\vec{0})$,
\item or $\Sigma$ yields entirely out this ball, $X(\Sigma) \subseteq \erre^{n+m}\setminus B^{n+m}_{\sqrt{\frac{n}{\lambda}}}(\vec{0})$
\end{enumerate}
     Then $\Sigma^n$ is compact and $X: \Sigma \rightarrow S^{n+m-1}(\sqrt{\frac{n}{\lambda}})$ is a minimal immersion.
\end{theorem}

\begin{proof}
Let us suppose first that $X(\Sigma) \subseteq B^{n+m}_{\sqrt{\frac{n}{\lambda}}}(\vec{0})$. Then $\sqrt{\frac{n}{\lambda}} \geq r(p) \,\, \forall p \in \Sigma$. Hence  we have that $ \Vert X^\bot \Vert^2 \leq  \Vert X \Vert^2 \leq  \frac{n}{\lambda}$. Then, using Lemma \ref{lemma}:
$$
 \triangle^\Sigma r^2(x)=2(n-\lambda \Vert X^\bot \Vert^2) \geq 0
 $$
On the other hand, as $X$ is proper and $\Sigma=X^{-1}(B^{n+m}_{\sqrt{\frac{n}{\lambda}}}(\vec{0})$, then $\Sigma$ is compact and hence, parabolic. In conclusion, $r^2(x)=R^2\,\,\forall x \in \Sigma$, for some $R \leq \sqrt{\frac{n}{\lambda}}$. But as $\Sigma$ is a $\lambda$-soliton for the MCF, then $R=\sqrt{\frac{n}{\lambda}}$ by Proposition \ref{teorSpherical}.

Let us suppose now that $X(\Sigma) \subseteq \erre^{n+m}\setminus B^{n+m}_{\sqrt{\frac{n}{\lambda}}}(\vec{0})$. This means that $\sqrt{\frac{n}{\lambda}} \leq r(p) \,\, \forall p \in \Sigma$ and hence, that there is not a point $p \in \Sigma$ such that $r(p) < \sqrt{\frac{n}{\lambda}}$. We apply Theorem \ref{teoCompactSphere} to conclude that $X: \Sigma \rightarrow S^{n+m-1}\left(\sqrt{\frac{\lambda}{n}}\right)$ is a compact minimal immersion.\end{proof}

Also as a corollary of Theorem \ref{teoCompactSphere}, we have the following characterization of minimal spherical immersions

\begin{corollary}
Let $X: \Sigma^n \rightarrow \erre^{n+m}$ be a complete and properly immersed $\lambda$-self-shrinker in $\erre^{n+m}$ 
for the MCF, with respect $\vec{0} \in \erre^{n+m}$.

Then, $X: \Sigma^n \rightarrow \erre^{n+m}$ is a compact minimal immersion of a round sphere of radius $\sqrt{\frac{n}{\lambda}}$ centered at $\vec{0}$ if and only if $\inf_{\Sigma}r=\sqrt{\frac{n}{\lambda}}$.
\end{corollary}
 
 \begin{remark}
Note that if either $\Sigma$ is confined into the ball $X(\Sigma) \subseteq B^{n+m}_{\sqrt{\frac{n}{\lambda}}}(\vec{0})$, or $\Sigma$ yields entirely out this ball, $X(\Sigma) \subseteq \erre^{n+m}\setminus B^{n+m}_{\sqrt{\frac{n}{\lambda}}}(\vec{0})$, then by Theorem \ref{ballcomp} we have that $\inf_{\Sigma}r=\sqrt{\frac{n}{\lambda}}$. Likewise, if either $\Sigma$ is confined into the ball $X(\Sigma) \subseteq B^{n+m}_{\sqrt{\frac{n}{\lambda}}}(\vec{0})$, or $\Sigma$ yields entirely out this ball, $X(\Sigma) \subseteq \erre^{n+m}\setminus B^{n+m}_{\sqrt{\frac{n}{\lambda}}}(\vec{0})$, then by Theorem \ref{ballcomp} we have that $\sup_{\Sigma}r=\sqrt{\frac{n}{\lambda}}$
\end{remark}

%%%%%%%%%%%%%%%%%%%%%%%%%%%%%%%%%%%%%%%%%%%%%
 \subsection{Comments on a classification of proper self-shrinkers for the MCF }\
 %%%%%%%%%%%%%%%%%%%%%%%%%%%%%%%%%%%%%%%%%
 \medskip
 
In \cite{CaoLi} it was proved the following classification result for self-shrinkers with polynomial volume growth. We remark here that in \cite{CZ} it was proved that properness of the immersion for self-shrinkers implies polynomial volume growth.

\begin{theorem}\label{Caoli}  
Let  $\Sigma^n\to \mathbb{R}^{n+m}$ be a complete $\lambda$-self-shrinker without boundary, polynomial volume growth and  bounded norm of the second fundamental form by
$$
\Vert A_\Sigma^{\mathbb{R}^{n+m}}\Vert^2\leq \lambda,
$$ 
 Then $\Sigma$ is one of the following:
 \begin{enumerate}
 \item $\Sigma$ is a round sphere $S^n(\sqrt{\frac{n}{\lambda}})$, (and hence $\Vert A_\Sigma^{\mathbb{R}^{n+m}}\Vert^2=\lambda$).
 \item $\Sigma$ is a cylinder $S^k(\sqrt{\frac{k}{\lambda}})\times\mathbb{R}^{n-k}$, (and hence $\Vert A_\Sigma^{\mathbb{R}^{n+m}}\Vert^2=\lambda$).
 \item $\Sigma$ is an hyperplane, (and hence $\Vert A_\Sigma^{\mathbb{R}^{n+m}}\Vert^2=0$).
 \end{enumerate}
 \end{theorem}
 
 %\begin{figure}\includegraphics[scale=0.45]{familia}
% \caption{Sphere $\mathbb{S}^2(1)$, plane $\mathbb{R}^2$ and cylinder $\mathbb{S}^2(\frac{1}{2})\times \mathbb{R}$ in $\mathbb{R}^3$. These  are the only $2$-self-shrinkers with $\Vert B\Vert^2\leq 2$ immersed in $\mathbb{R}^3$ with polynomial volume growth. The sphere $\mathbb{S}^2(1)$ divides  $\mathbb{R}^2$ and the cylinder $\mathbb{S}^2(\frac{1}{2})\times \mathbb{R}$ in three parts. One part at the intersection with $\mathbb{S}^2(1)$, an other part inside of  $\mathbb{S}^2(1)$ and the other part outside of $\mathbb{S}^2(1)$.}\label{figure1}\end{figure}

We want to draw attention at this point on the following notion of {\em separation} of a submanifold:

\begin{definition}
We say that the sphere $S^{n+m-1}(\sqrt{\frac{n}{\lambda}})$ \emph{separates} the $\lambda$-self-shrinker $X:\Sigma\to \mathbb{R}^{n+m}$ if 
$$
D_{\sqrt{\frac{n}{\lambda}}}=\left\{p\in \Sigma\, :\, \Vert X(p)\Vert <\sqrt{\frac{n}{\lambda}}\right\}\neq \emptyset, 
$$
and 
$$
\Sigma\setminus\overline{D_{\sqrt{\frac{n}{\lambda}}}}=\left\{p\in \Sigma\, :\, \Vert X(p)\Vert >\sqrt{\frac{n}{\lambda}}\right\}\neq \emptyset,
$$
\end{definition}

\begin{remark}\label{remn}\

When we consider  any of the three proper and complete examples with  $\Vert A_\Sigma^{\mathbb{R}^{n+m}}\Vert^2\leq \lambda$ in Theorem \ref{Caoli}, the critical sphere of radius $\sqrt{\frac{n}{\lambda}}$ in $\mathbb{R}^{n+m}$ separates the self-shrinker $\Sigma$ unless $\Sigma$ is itself a round sphere $S^n(\sqrt{\frac{n}{\lambda}})$ and $\Vert A_\Sigma^{\mathbb{R}^{n+m}}\Vert^2=\lambda$. On the other hand, Theorem \ref{ballcomp} is telling us that {\em non-separated} $\lambda$-self-shrinkers by the critical sphere of radius $\sqrt{\frac{n}{\lambda}}$ must be isometrically immersed in $S^n(\sqrt{\frac{n}{\lambda}})$ as compact and minimal submanifolds.
\end{remark}
%See figure \ref{figure1}. }

In Theorem \ref{teo7.7v24} of this section we will prove that the fact described in Remark \ref{remn} is still true when the squared norm of the second fundamental form of $\Sigma$ is bounded above by the greather constant $\frac{5}{3}\lambda$. More precisely, in Theorem \ref{teo7.7v24} we will prove that the sphere of radius $\sqrt{\frac{n}{\lambda}}$ separates any $\lambda$-self-shrinker properly immersed in $\mathbb{R}^{n+m}$ with $\Vert A_\Sigma\Vert^2<\frac{5}{3}\lambda$ unless the self-shrinker is just the the $n$-sphere of radius $\sqrt{\frac{n}{\lambda}}$.

To prove Theorem \ref{teo7.7v24} we will make use of the classification provided by J. Simon, and S.S. Chern, M.P. Do Carmo and  S. Kobayashi, for compact minimal immersions in the sphere, (see \cite{Si}, \cite{CDK}, \cite{BYCh}), refined later by A.M. Li and J.M. Li, ( see \cite{LiLi}). These results can be summarized in the following statement:

\begin{theorem}[Simon-Do Carmo-Chern-Kobayashi Classification after Li and Li]\label{theoSiLi}\

Let  $\varphi: (\Sigma^n,\widetilde g)\to (S^{n+m-1}(1),g_{S^{n+m-1}(1)})$ be a compact and minimal isometric immersion.
\begin{enumerate}
\item If $m=1$ or $m=2$, let us suppose that $\Vert \widetilde{A}_\Sigma^{S^{n+m-1}(1)}\Vert^2 \leq \frac{n}{2-\frac{1}{m-1}}=\frac{m-1}{2m-3} n$. Then,  

\begin{enumerate}
\item either $\Vert \widetilde{A}_\Sigma^{S^{n+m-1}(1)}\Vert^2=0$ and $(\Sigma^n,\widetilde g)$ is isometric to $S^n(1)$, 
\item or either, (in case $m=2$),  $\Vert \widetilde{A}_\Sigma^{S^{n+m-1}(1)}\Vert^2=n$ and  $(\Sigma^n,\widetilde g)$ is isometric to a generalized Clifford torus $\Sigma^n=S^{k}(\sqrt{\frac{k}{n}})\times S^{n-k}(\sqrt{\frac{n-k}{n}})$ immersed as an hypersurface in $S^{n+1}(1)$.
\end{enumerate}

\item If $m \geq 3$, let us suppose that  $\Vert \widetilde{A}_\Sigma^{S^{n+m-1}(1)}\Vert^2 \leq \frac{2n}{3}$. Then,

\begin{enumerate}
\item either $(\Sigma^n,\widetilde g)$ is isometric to $S^n(1)$, and $\Vert \widetilde{A}_\Sigma^{S^{n+m-1}(1)}\Vert^2=0$
\item or when $n=2$ and $m=3$, then $(\Sigma^n,\widetilde g)$ is isometric to the Veronese surface $\Sigma^2=\erre P^2(\sqrt{3})$ in $S^{4}(1)$, and $\Vert \widetilde{A}_\Sigma^{S^{n+m-1}(1)}\Vert^2=\frac{4}{3}$. 
\end{enumerate}
\end{enumerate}
\end{theorem}
 
\begin{remark}
 It is easy to check that the bound for the squared norm of the second fundamental form $\frac{2}{3}n$,  used in \cite{LiLi} and which do not depends on the codimension $m$, is bigger or equal than the bound $\frac{m-1}{2m-3} n$ used in \cite{Si}, \cite{CDK}, \cite{BYCh}, when $m \geq 3$. In fact, for all $n >0$, the values are equal when $m=3$ and $\frac{2}{3} n > \frac{m-1}{2m-3} n$ when $m > 3$. 
 \end{remark}

Let us consider now $X:(\Sigma,g)\to (\mathbb{R}^{n+m},g_{\rm can})$ a complete and properly immersed $\lambda$-self-shrinker in $\erre^{n+m}$. By Theorem \ref{ballcomp}, if the critical sphere of radius $\sqrt{\frac{n}{\lambda}}$ does not separate $X(\Sigma)$, then $\Sigma$ is therefore compact and is minimally immersed in the round sphere $S^{n+m-1}(\sqrt{\frac{n}{\lambda}})$ centered at $\vec{0}$.

We are going to present some computations to rescale the immersion $X$ in order to apply Theorem \ref{theoSiLi} to this situation. For that, we are interested in to know what is the relation between the squared norm $\Vert \widetilde A_\Sigma^{S^{n+m-1}(1)}\Vert^2$, (corresponding to the isometric immersion $\widetilde{X}: (\Sigma,\widetilde{g})\to (S^{n+m-1}(1),g_{S^{n+m-1}(1)})$) and the squared norm $\Vert A_\Sigma^{\mathbb{R}^{n+m}}\Vert^2$, (which corresponds to the isometric immersion $X: (\Sigma,g)\to (S^{n+m-1}(\sqrt{\frac{n}{\lambda}}),g_{S^{n+m-1}(\sqrt{\frac{n}{\lambda}})})$).

The first thing to do that is to relate the metrics on $\Sigma$, $g$ and $\widetilde{g}$. Note that, given the immersion $X: (\Sigma,g)\to (\mathbb{R}^{n+m},g_{\rm can})$, the rescaled map
$$
\widetilde{X}:\Sigma\to \mathbb{R}^{n+m},\quad p \to \widetilde{X}(p):=\sqrt{\frac{\lambda}{n}}X(p)
$$
sends  $\Sigma$ into $S^{n+m-1}(1)$, with codimension $m-1$. Therefore,
$$
\widetilde{X}:(\Sigma, \frac{\lambda}{n}g)\to (\mathbb{R}^{n+m},g_{\rm can})
$$ 
is an isometric immersion, and in fact, $\widetilde{X}:(\Sigma,\frac{\lambda}{n}g)\to (S^{n+m-1}(1),g_{S^{n+m-1}(1)})$ realizes as a minimal immersion if $X$ is minimal. Hence $\widetilde{g}=\frac{\lambda}{n}g$.

Moreover, it is straightforward to check from this that:
% if $\widetilde B_\Sigma^{S^{n+m-1}(1)}$ and  $B_\Sigma^{S^{n+m-1}(\sqrt{\frac{n}{\lambda}})}$ denotes the second fundamental forms of $\widetilde{X}: (\Sigma,\frac{\lambda}{n}g)\to (S^{n+m-1}(1),g_{S^{n+m-1}(1)})$ and $X: (\Sigma,g)\to (S^{n+m-1}(\sqrt{\frac{n}{\lambda}}),g_{S^{n+m-1}(\sqrt{\frac{n}{\lambda}})})$ respectively, we have 
$$
\left\Vert \widetilde A_\Sigma^{S^{n+m-1}(1)}\right\Vert^2=\frac{n}{\lambda}\left\Vert A_\Sigma^{S^{n+m-1}(\sqrt{\frac{n}{\lambda}})}\right\Vert^2
$$

\noindent and that
$$
%\begin{aligned}
\left\Vert A_\Sigma^{\mathbb{R}^{n+m}}\right\Vert^2
%=&\sum_{ij}^n\left\Vert A_\Sigma^{\mathbb{R}^{n+m}}(E_i,E_j)\right\Vert^2\\
%=&\sum_{ij}^n\left\Vert A_\Sigma^{S^{n+m-1}(\sqrt{\frac{n}{\lambda}})}(E_i,E_j)\right\Vert^2+\sum_{ij}^n\left\Vert B_{S^{n+m-1}}^{\mathbb{R}^{n+m}(\sqrt{\frac{n}{\lambda}})}(E_i,E_j)\right\Vert^2\\
=\left\Vert A_\Sigma^{S^{n+m-1}(\sqrt{\frac{n}{\lambda}})}\right\Vert^2+\lambda.
%\end{aligned}
$$
\noindent Then we conclude 
\begin{equation}\label{eq7.24v22}
\left\Vert \widetilde A_\Sigma^{S^{n+m-1}(1)}\right\Vert^2=\frac{n}{\lambda}\left\Vert A_\Sigma^{\mathbb{R}^{n+m}}\right\Vert^2-n.
\end{equation}

\noindent With this last equation in hand, it is obvious that  the bound for the squared norm of the second fundamental form 
$$
\left\Vert \widetilde A_\Sigma^{S^{n+m-1}(1)}\right\Vert^2\leq \frac{n}{2-\frac{1}{m-1}}.
$$
\noindent is equivalent to the bound $\Vert A_\Sigma^{\erre^{n+m}}\Vert ^2 \leq \frac{3m-4}{2m-3}\lambda$.

Moreover, and in the same way, the bound for the squared norm of the second fundamental form  given by $\left\Vert \widetilde A_\Sigma^{S^{n+m-1}(1)}\right\Vert^2\leq \frac{2n}{3}$ is equivalent to the bound $\Vert A_\Sigma^{\erre^{n+m}}\Vert ^2 \leq \frac{5}{3}\lambda$.

The previous comments allow us to state the following Theorem,
\begin{theorem}\label{teo7.7v24}
Let $X:\Sigma^n\to \mathbb{R}^{n+m}$ be a complete, connected and properly immersed $\lambda$-self-shrinker for the MCF with respect to $\vec{0}\in\mathbb{R}^{n+m}$. Let us suppose that 
  $$\Vert A_\Sigma^{\mathbb{R}^{n+m}}\Vert^2< \frac{5}{3}\lambda$$

 Then, either
\begin{enumerate}
\item $\Sigma^n$ is isometric to $S^n\left(\sqrt{\frac{n}{\lambda}}\right)$ and $\Vert A_\Sigma^{\mathbb{R}^{n+m}}\Vert^2=\lambda$,
\item or, the sphere $S_{\sqrt{\frac{n}{\lambda}}}^{n+m-1}(\vec{0})$ of radius $\sqrt{\frac{n}{\lambda}}$ centered at $\vec{0}\in\mathbb{R}^{n+m}$ separates $X(\Sigma)$. 
\end{enumerate}
\end{theorem}
\begin{remark}
The bound $\frac{5}{3}\lambda$ is optimal in the following sense: the Veronese surface $\Sigma^2=\erre P^2(\sqrt{3})$ in $\erre^5$ satisfies that  $\Vert A_\Sigma^{\mathbb{R}^{n+m}}\Vert^2=\frac{5}{3}\lambda$ and it is not separated by sphere $S_{\sqrt{\frac{2}{\lambda}}}^{4}(\vec{0})$ of radius $\sqrt{\frac{2}{\lambda}}$ centered at $\vec{0}\in\mathbb{R}^{5}$.
\end{remark}
\begin{proof}
We are going to see first that, if $(1)$ is not satisfied, then it is satisfied $(2)$. Namely, the fact that $\Sigma^n$ is not isometric to $S^n\left(\frac{n}{\lambda}\right)$ or  $\Vert A_\Sigma^{\mathbb{R}^{n+m}}\Vert^2\neq \lambda$, implies that the sphere $S_{\sqrt{\frac{n}{\lambda}}}^{n+m-1}(\vec{0})$ of radius $\sqrt{\frac{n}{\lambda}}$ centered at $\vec{0}\in\mathbb{R}^{n+m}$ separates $X(\Sigma)$. 

To see this, let us suppose that this sphere does not separates $X(\Sigma)$. Then, by Theorem \ref{ballcomp}, $X: (\Sigma,g) \to (S^{n+m-1}(\sqrt{\frac{n}{\lambda}}), g_{S^{n+m-1}(\sqrt{\frac{n}{\lambda}})})$ is a  compact and minimal immersion.
 Hence:
 \begin{enumerate}
 	\item If $m=1$, $\Sigma^n$ is isometric to $S^n(\sqrt{\frac{n}{\lambda}})$ because $X$ is a Riemannian covering and $S^n(\sqrt{\frac{n}{\lambda}})$ is simply connected, following the same arguments than in Corollaries \ref{cor10} and \ref{cor11}. Hence  $\Vert A_\Sigma^{\mathbb{R}^{n+m}}\Vert^2=\lambda$. But this is a contradiction with the assumption that $\Sigma^n$ is not isometric to $S^n\left(\sqrt{\frac{n}{\lambda}}\right)$ or  $\Vert A_\Sigma^{\mathbb{R}^{n+m}}\Vert^2\neq \lambda$.
 	\item If $m=2$, since  $\Vert A_\Sigma^{\mathbb{R}^{n+m}}\Vert^2<\frac{5}{3}\lambda< 2\lambda$ then, applying Theorem \ref{theoSiLi}, either 
 	\begin{enumerate}
 	\item $\Sigma$ is isometric to $S^n(\sqrt{\frac{n}{\lambda}})$ and  $\Vert A_\Sigma^{\mathbb{R}^{n+m}}\Vert^2=\lambda$. But this is a contradiction with the assumption that $\Sigma^n$ is not isometric to $S^n\left(\frac{n}{\lambda}\right)$ or  $\Vert A_\Sigma^{\mathbb{R}^{n+m}}\Vert^2\neq \lambda$.
 	\item or, $\Sigma$ is isometric to the Clifford torus $S^k\left(\sqrt{\frac{k}{n\lambda}}\right)\times S^{n-k}\left(\sqrt{\frac{n-k}{n\lambda}}\right)$ and $\Vert A_\Sigma^{\mathbb{R}^{n+m}}\Vert^2=2\lambda$. But this is a contradiction with the hypothesis of norm of second fundamental form bounded from above by  $\Vert A_\Sigma^{\mathbb{R}^{n+m}}\Vert^2< 2 \lambda$ . 
 	\end{enumerate}
 	\item If $m=3$, since  $\Vert A_\Sigma^{\mathbb{R}^{n+m}}\Vert^2< \frac{3}{5} \lambda$ then applying Theorem \ref{theoSiLi},  either 
 	\begin{enumerate}
 	\item $\Sigma$ is isometric to $S^n(\sqrt{\frac{n}{\lambda}})$ and  $\Vert A_\Sigma^{\mathbb{R}^{n+m}}\Vert^2=\lambda$. But this is a contradiction with the assumption that $\Sigma^n$ is not isometric to $S^n\left(\frac{n}{\lambda}\right)$ or  $\Vert A_\Sigma^{\mathbb{R}^{n+m}}\Vert^2\neq \lambda$.
 	\item or, $\Sigma$ is isometric to the Veronese surface  in $S^4(\sqrt{\frac{n}{\lambda}})$  and $\Vert A_\Sigma^{\mathbb{R}^{n+m}}\Vert^2=\frac{3}{5}\lambda$. But this is a contradiction with the hypothesis of $\Vert A_\Sigma^{\mathbb{R}^{n+m}}\Vert^2< \frac{3}{5} \lambda$ . 
 	\end{enumerate}
\item 	If $m>3$, then, applying Theorem \ref{theoSiLi}, $\Sigma$ should be isometric to $S^n(\sqrt{\frac{n}{\lambda}})$ and  $\Vert A_\Sigma^{\mathbb{R}^{n+m}}\Vert^2=\lambda$. But again this is a contradiction with the assumption that $\Sigma^n$ is not isometric to $S^n\left(\frac{n}{\lambda}\right)$ or  $\Vert A_\Sigma^{\mathbb{R}^{n+m}}\Vert^2\neq \lambda$.
 \end{enumerate}
 
Conversely, if the sphere $S_{\sqrt{\frac{n}{\lambda}}}^{n+m-1}(\vec{0})$ of radius $\sqrt{\frac{n}{\lambda}}$ centered at $\vec{0}\in\mathbb{R}^{n+m}$ does not separate $X(\Sigma)$, then, as we have argumented before, by Theorem \ref{ballcomp}, $X: (\Sigma,g) \to (S^{n+m-1}(\sqrt{\frac{n}{\lambda}}), g_{S^{n+m-1}(\sqrt{\frac{n}{\lambda}})})$ is a compact and minimal immersion, and hence $\widetilde{X}:(\Sigma,\frac{\lambda}{n}g)\to (S^{n+m-1}(1),g_{S^{n+m-1}(1)})$ realizes as a minimal immersion, with second fundamental form in the sphere satisfying $\left\Vert \widetilde A_\Sigma^{S^{n+m-1}(1)}\right\Vert^2 < \frac{2n}{3}$ because by hypothesis $\Vert A_\Sigma^{\mathbb{R}^{n+m}}\Vert^2< \frac{5}{3}\lambda$. Therefore we apply Theorem \ref{theoSiLi} to conclude that
\begin{enumerate}
\item  $\Sigma^n$ should be isometric to $S^n(\sqrt{\frac{n}{\lambda}})$ and
\item  $\Vert A_\Sigma^{\mathbb{R}^{n+m}}\Vert^2=\lambda$
\end{enumerate}
\end{proof}
\medskip
\section{Mean Exit Time, and Volume of IMCF-solitons}\label{meanexitInverse}\
%%%%%%%%%%%%

%%%%%%%%%%%%%%%%%%%%%%%%%%%%
\subsection{Mean Exit time on Solitons for IMCF}\label{volumeS}\
%%%%%%%%%%%%%%%%%%%%%%%%%%%%%%%%

We start studying the Mean Exit Time function on properly immersed solitons for IMCF $X: \Sigma^n \rightarrow \erre^{n+m}$.

As in Subsection \ref{moments}, let us consider the Poisson problem defined on extrinsic $R$-balls $D_R \subseteq \Sigma$
\begin{equation}  \label{poisson2}
\begin{aligned}
\Delta^\Sigma E+1 &= 0\,\,\, \text{on}\,\,\, D_R, \\
E\vert_{\partial D_R} &=0.
\end{aligned}
\end{equation}

We saw that the solution of the Poisson problem (\ref{poissoneuclidean}) on a geodesic $R$- ball $B^n(R)$ in $\erre^n$
is given by the radial function $ E^{0,n}_R(r)= \frac{R^2-r^2}{2n}$.

As in Subsection \S.\ref{volumeSS}, we shall consider the transplanted radial solution of (\ref{poisson}) $\bar E_R(r)$ to the extrinsic ball by mean the extrinsic distance function, so we have $\bar E_R: D_R \rightarrow \erre$ defined as $\bar E_R(p):= \bar E_R(r(p)) \,\,\forall p \in D_R$. Our first result here is again a comparison for the Mean Exit Time function:

\begin{proposition}
Let $X: \Sigma^n \rightarrow \erre^{n+m}$ be a complete properly immersed soliton in $\erre^{n+m}$ for the IMCF, with constant velocity $C \neq 0$ and  with respect $\vec{0} \in \erre^{n+m}$. Let us suppose that $X(\Sigma) \not\subseteq S^{n+m-1}(R)$ for any radius $R>0$. Given the extrinsic ball $D_R(\vec{0})=\Sigma \cap B_R^{n+m}(\vec{0})$, we have that the mean exit time function on $D_R$, $E_R$, satisfies
\begin{equation}\label{eq.21}
E_R(x)=\frac{C n}{C n-1}\bar {E}_R(x)  \,\,\,\forall x \in D_R
\end{equation}
\end{proposition}

\begin{proof}
We have, as $\bar{E}_R(x):=E^{0,n}_R(r(x))= \frac{R^2-r(x)^2}{2n}$ and applying Lema \ref{lemma}, that, on $D_R$
\begin{equation}\label{eq.22}
\begin{aligned}
\Delta^\Sigma \bar{E}_R&=\left(\bar{E}''_R(r)-\bar{E}'_R(r)\frac{1}{r}\right)\Vert \nabla^\Sigma r\Vert^2\\&+\Bar{E}'_r(r)\Big(n\frac{1}{r}+\langle \nabla^{\erre^{n+m}} r, \vec{H}_\Sigma\rangle\Big)=-1-\frac{1}{n}\langle r \nabla^{\erre^{n+m}} r, \vec{H}_\Sigma\rangle
\end{aligned}
\end{equation}
On the other hand, $X(p)=r(p)\nabla^{\erre^{n+m}}(p)$ for all $p \in \Sigma$, being $X(p)$ the position vector of $p$ in $\erre^{n+m}$. And, moreover, as we have that $\frac{\vec{H}_\Sigma(p)}{\Vert\vec{H}_\Sigma(p)\Vert^2}=-C X^{\bot}(p)$, then 
\begin{equation}\label{eq.23}
\begin{aligned}
\langle r \nabla^{\erre^{n+m}} r, \vec{H}_\Sigma\rangle&=\langle X, \vec{H}_\Sigma\rangle=\langle X, -C\Vert \vec{H}_\Sigma\Vert^2 X^\bot \rangle\\&=-C\Vert \vec{H}_\Sigma\Vert^2\Vert X^\bot\Vert^2=-\frac{1}{C}.
\end{aligned}
\end{equation}
\noindent
Equation (\ref{eq.22}) becomes
\begin{equation}\label{eq.24}
%\begin{aligned}
\Delta^\Sigma \bar{E}_R=
-1+\frac{1}{Cn}=\frac{1-Cn}{Cn}
%\end{aligned}
\end{equation}
Therefore,
\begin{equation}\label{eq.25}
\begin{aligned}
\Delta^\Sigma\frac{Cn}{Cn-1} \bar{E}_R&=\frac{Cn}{Cn-1}\Delta^\Sigma \bar{E}_R=\frac{Cn}{Cn-1}\frac{1-Cn}{Cn}\\&=-1=\Delta^\Sigma E_R\,\,\,\text{on} \,\,\,D_R
\end{aligned}
\end{equation}
\noindent and, applying the Maximum Principle, $$\frac{Cn}{Cn-1} \bar{E}_R=E_R\,\,\,\text{on} \,\,\,D_R$$
\end{proof}

As a consequence, we obtain again Corollary \ref{teorSuffParabIMCF}:
\begin{corollary}
Let $X: \Sigma^n \rightarrow \erre^{n+m}$ be a complete and non-compact, properly immersed soliton in $\erre^{n+m}$ for the IMCF, with constant velocity $C \neq 0$ and  with respect $\vec{0} \in \erre^{n+m}$.  Then
$$C \notin (0,\frac{1}{n}]$$

\end{corollary}

We finalize this subsection with a characterization of solitons for the IMCF in terms of the mean exit time function
\begin{theorem}\label{solcarac}
Let $X: \Sigma^n \rightarrow \erre^{n+1}$ be a proper immersion. Let us suppose that $X(\Sigma) \not\subseteq S^{n}(R)$ for any radius $R>0$. Then, if for all extrinsic $R$-balls $D_R(\vec{0})$, we have that  $E_R=\alpha \bar{E}_R$, with $\alpha \neq 1$ and $\alpha \neq 0$, then $X$ is a soliton for the IMCF  with respect $\vec{0} \in \erre^{n+1}$, with velocity $C=-\frac{\alpha}{\alpha-1}\frac{1}{n}$. Hence, if $\alpha  \in (1, \infty) $, then $X$ is a self-shrinker and if $\alpha \in (0,1)$, then $X$ is a self-expander.
\end{theorem}
\begin{proof}
We have, as $\bar{E}_R(x):=E^{0,n}_R(r(x))= \frac{R^2-r(x)^2}{2n}$ and applying Lema \ref{lemma}, that, on $D_R$, for all $R>0$,
\begin{equation}\label{e1}
\Delta^\Sigma \bar{E}_R=-1-\frac{1}{n}\langle X, \vec{H}_\Sigma\rangle
\end{equation}

Hence, as we are assuming that $E_R=\alpha\bar{E}_R$ for all $R>0$, we have
\begin{equation}\label{e2}
\Delta^\Sigma \alpha\bar{E}_R=-\alpha-\frac{\alpha}{n}\langle X, \vec{H}_\Sigma\rangle=-1
\end{equation}

Therefore, on $\Sigma$,
\begin{equation}\label{e3}
\langle X, \vec{H}_\Sigma\rangle=\langle X^{\bot}, \vec{H}_\Sigma\rangle=\frac{1-\alpha}{\alpha}n
\end{equation}
\noindent so $\Vert \vec{H}\Vert \neq 0$.

But $\vec{H}_\Sigma= h \nu$ where $\nu$ is the unit normal vector field pointed outward to $\Sigma$, so 
$$\langle X, \vec{H}_\Sigma\rangle=\langle X, \nu\rangle h=\frac{1-\alpha}{\alpha}n$$
\noindent and therefore
\begin{equation}\label{e4}
\frac{1}{n}\frac{\alpha}{1-\alpha}\langle X, \nu\rangle \nu =\frac{1}{h} \nu
\end{equation}

Hence, as $X^{\bot}=\langle X^{\bot}, \nu \rangle \nu=\langle X, \nu \rangle \nu$
\begin{equation}\label{e5}
\frac{\vec{H}}{\Vert \vec{H}\Vert^2}=\frac{1}{h} \nu=\frac{\alpha}{1-\alpha}\frac{1}{n}\langle X, \nu \rangle \nu=\frac{\alpha}{1-\alpha}\frac{1}{n}X^\bot
\end{equation}

\noindent and $X$ is a soliton with $C=-\frac{\alpha}{1-\alpha}\frac{1}{n}$.
\end{proof}
\begin{remark}
Note that $\alpha \neq 0, 1$. If $\alpha=0$, then $E_R=0$ for all radius $R>0$, so $\Sigma$ reduces to a point. On the other hand, if $\alpha=1$, then $\Sigma$ is minimal in $\erre^{n+1}$, (see \cite{Ma}), and hence $X$ cannot be a soliton for the IMCF, (see Remark \ref{minimalIMCF}).
\end{remark}

%%%%%%%%%%%%%%%%%%%%%%%%%%%%
\subsection{Volume of Solitons for IMCF}\label{volumS}\
%%%%%%%%%%%%%%%%%%%%%%%%%%%%%%%%
As a consequence or the proof above, and using the Divergence theorem we have the following result:

\begin{theorem}\label{isopSol}
Let $X: \Sigma^n \rightarrow \erre^{n+m}$ be a complete properly immersed soliton in $\erre^{n+m}$ for the IMCF, with constant velocity $C \neq 0$ and  with respect $\vec{0} \in \erre^{n+m}$. Let us suppose that $X(\Sigma) \not\subseteq S^{n+m-1}(R)$ for any radius $R>0$. Given the extrinsic ball $D_R(\vec{0})$, we have 
\begin{equation} \label{isopComp3}
\frac{\Vol(\partial D_R)}{\Vol(D_R)} \geq \frac{Cn-1}{Cn}\frac{\Vol(S^{n-1}(R))}{\Vol(B^{n}(R))}\,\,\,\,\,\textrm{for all}\,\,\,
R>0  \quad 
\end{equation}
%where 
%\begin{equation}\label{factor}
%1-\frac{1}{n\lambda}\frac{\int_{D_r} H_\Sigma^2}{\Vol(D_r)} \geq 0 \,\,\forall r >0
\end{theorem}
\begin{proof}
Integrating on the extrinsic ball $D_R$ the equality $\Delta^\Sigma\frac{Cn}{Cn-1} \bar{E}_R=-1$ and applying Divergence theorem as in Theorem \ref{isopShri} we obtain, as $C \in \erre \sim[0,\frac{1}{n}]$ :
\begin{equation}
\begin{aligned}\label{eq.26}
\Vol(D_R)&=\int_{D_R} -\Delta^\Sigma\frac{Cn}{Cn-1} \bar{E}_R=-\frac{Cn}{Cn-1}\bar{E}_R'(R)\int_{\partial D_R} \Vert \nabla^\Sigma r\Vert d\sigma\\& \leq \frac{Cn}{Cn-1}\frac{\Vol(B^{n}(R))}{\Vol(S^{n-1}(R))} \Vol(\partial D_R)
\end{aligned}
\end{equation}
\end{proof}
\begin{remark}
Equality in inequality (\ref{isopComp3}) for all radius $R \leq R_0$ implies that the inequality $\int_{\partial D_R} \Vert \nabla^\Sigma r\Vert d\sigma \leq \Vol(\partial D_R)$ becomes an equality for all $R \leq R_0$. This implies that $\Vert \nabla^\Sigma r\Vert=1=\Vert \nabla^{\erre^{n+m}} r\Vert$ in the extrinsic ball $D_{R_0}$, so  $\nabla^\Sigma r=\nabla^{\erre^{n+m}} r$ in $D_{R_0}$ and $\Sigma$ is totally geodesic in $D_{R_0}$. Hence, $\vec{H}_\Sigma=\vec{0}$ in $D_{R_0}$, which is not compatible with the fact that $X: \Sigma^n \rightarrow \erre^{n+m}$ be a properly immersed soliton in $\erre^{n+m}$ for the IMCF. Therefore, if $X: \Sigma^n \rightarrow \erre^{n+m}$ is a properly immersed soliton in $\erre^{n+m}$ for the IMCF, then inequality (\ref{isopComp3}) must be strict.
\end{remark}

\begin{corollary}\label{monotonicity}
Let $X: \Sigma^n \rightarrow \erre^{n+m}$ be a properly immersed soliton in $\erre^{n+m}$ for the IMCF, with constant velocity $C \neq 0$ and  with respect $\vec{0} \in \erre^{n+m}$. Let us suppose that $X(\Sigma) \not\subseteq S^{n+m-1}(R)$ for any radius $R>0$. Let us define the {\em volume growth} function $$f(t):=\frac{\Vol(D_t)}{\Vol(B^{n}(t))^{\frac{C n-1}{Cn}}}$$ Then, given $r_1>0$, $f(t)$ is non decreasing for all $t \geq r_1 >0$.
\end{corollary}
\begin{proof}
As $\frac{d}{dt}\Vol(D_t) \geq \Vol(\partial D_t)$ by the co-area formula, we have, applying Theorem \ref{isopSol},
$$\frac{d}{dt}\ln f(t)\geq \frac{\Vol(\partial D_t)}{\Vol(D_t)} -\frac{C n-1}{C n} \frac{\Vol(\partial S^{n-1}(t))}{\Vol(B^{n}(t))}\geq 0$$\end{proof}

  %%%%%%%%%%%%%%%%

\end{document}